\documentclass[reqno, a4paper, 10pt]{amsart}
\usepackage{amssymb,url, color, mathrsfs, multirow, makecell, array, caption}
\usepackage[colorlinks=true, bookmarks=true, pdfstartview=FitH, pagebackref=true, linktocpage=true]{hyperref}
\usepackage[short,nodayofweek]{datetime}
\usepackage[pagewise,running,mathlines,displaymath, switch]{lineno}
\usepackage{times}
\usepackage{indentfirst, tabularx}
\usepackage[table]{xcolor}
\usepackage{float, hhline, colortbl}
\usepackage{graphicx}
\usepackage{longtable}
\usepackage{empheq}
\usepackage{aligned-overset}
\usepackage{kpfonts, baskervald}

\usepackage[framemethod=TikZ]{mdframed}

\usepackage[utf8]{inputenc}

\surroundwithmdframed[
topline=false,
rightline=true,
bottomline=false,
leftmargin=\parindent,
skipabove=\medskipamount,
skipbelow=\medskipamount
]{theorem}

\theoremstyle{plain}
\numberwithin{equation}{section}
\newtheorem{theorem}{Theorem}[section]
\newtheorem{lemma}[theorem]{Lemma}
\newtheorem{corollary}[theorem]{Corollary}
\newtheorem{proposition}[theorem]{Proposition}
\theoremstyle{remark}

\newtheorem{remark}[theorem]{Remark}

\hypersetup{
	pdftitle={A Liouville type result for fractional GJMS equations on higher dimensional spheres},
	pdfauthor={Q.N.T. L\^e, Q.A. Ng\^o, T.T. Nguyễn},
	colorlinks = true,
	linkcolor = magenta,
	citecolor = blue,
}

\def\ve{\varepsilon}
\def\R{\mathbf R}
\def\N{\mathbb N}
\newcommand{\gjms}{\, {\mathbf P}_n^{2s}}
\newcommand{\Q}{\, Q_n^{2s}}
\renewcommand{\S}{\mathbb S}

\numberwithin{equation}{section}

\makeatletter
\def\@cite#1#2{[\textbf{#1}\if@tempswa, #2\fi]}
\makeatother

\title[Fractional GJMS equations on higher dimensional spheres]{A Liouville type result for fractional GJMS equations on higher dimensional spheres}

\def\cfac#1{\ifmmode\setbox7\hbox{$\accent"5E#1$}\else\setbox7\hbox{\accent"5E#1}\penalty 10000\relax\fi\raise 1\ht7\hbox{\lower1.05ex\hbox to 1\wd7{\hss\accent"13\hss}}\penalty 10000\hskip-1\wd7\penalty 10000\box7 }

\def\cftil#1{\ifmmode\setbox7\hbox{$\accent"5E#1$}\else\setbox7\hbox{\accent"5E#1}\penalty 10000\relax\fi\raise 1\ht7\hbox{\lower1.1ex\hbox to 1\wd7{\hss\accent"7E\hss}}\penalty 10000\hskip-1\wd7\penalty 10000\box7 } 

\author[Q.N.T. L\^e]{Qu\`ynh N.T. L\^e}
\address[Q.N.T. L\^e]{
University of Economics \& Business, Vietnam National University, Hanoi, Vietnam}
\email{\href{mailto: Q.N.T. L\^e <ngocquynhlt@vnu.edu.vn>}{ngocquynhlt@vnu.edu.vn}
}

\author[Q.A. Ng\^o]{Qu\'{\^o}c Anh Ng\^o}
\address[Q.A. Ng\^o]{
	University of Science, Vietnam National University, Hanoi, Vietnam
	\\
	ORCID iD: 0000-0002-3550-9689}
\email{\href{mailto: Q.A. Ng\^o <nqanh@vnu.edu.vn>}{nqanh@vnu.edu.vn}
}

\author[T.-T. Nguyen]{Tien-Tai Nguyen}
\address[T.-T. Nguyen]{
University of Science, Vietnam National University, Hanoi, Vietnam\\
ORCID iD: 0000-0002-5865-8134}
\email{\href{mailto: T.-T. Nguyen <nttai.hus@vnu.edu.vn>}{nttai.hus@vnu.edu.vn}
}

\allowdisplaybreaks

\begin{document}
	
	\begin{abstract}
Let $n$ be an integer and $s$ be a real number such that $n > 2s \geq 2$. Inspired by the perturbation approach initiated by F. Hang and P. Yang (\textit{Int. Math. Res. Not. IMRN}, 2020), we are interested in non-negative, smooth solution $v$ to the following higher-order fractional equation
\[
\gjms (v) = \Q (\ve v+v^\alpha ) 
\]
on $\S^n$ with $0<\alpha \leq (n+2s)/(n-2s)$, and $\ve \geq 0$. Here $\gjms$ is the fractional GJMS type operator of order $2s$ on $\S^n$ and $\Q =\gjms(1)$ is constant. We show that if $\ve>0$ and $0<\alpha \leq (n+2s)/(n-2s)$, then any positive, smooth solution $v$ to the above equation must be constant. The same result remains valid if $\ve=0$ but with $0<\alpha < (n+2s)/(n-2s)$. As a by-product, with $0<\alpha\leq (n+2s)/(n-2s)$, we compute the sharp constant of  the subcritical/critical Sobolev inequalities 
\[
\int_{\S^n} v \gjms (v) d\mu_{g_{\S^n}}
\geq \frac{\Gamma (n/2 + s)}{\Gamma (n/2 - s )} | \S^n|^\frac{\alpha-1}{\alpha+1}
\Big( \int_{\S^n} v^{\alpha+1} d\mu_{g_{\S^n}} \Big)^\frac{2}{\alpha+1}
\]
for the GJMS operator $\gjms$ on $\S^n$ and for all non-negative functions $v\in H^s(\S^n)$.
\end{abstract}
	
\date{\bf \today \, at \currenttime}
	
	
	
\maketitle
	
\section{Data availability}
No data was used for the research described in the article	
\section{Introduction}
Let $n$ be an integer and $s$ be a real number such that $n > 2s \geq 2$. Let $0 < \alpha \leq (n+2s)/(n-2s)$ and $\varepsilon > 0$. In this paper, we are interested in Liouville type results for the following higher-order fractional equation
\begin{equation}\label{eq-MAIN}
\gjms (v) = \Q (\ve v+v^\alpha ) \quad \text{on } \S^n,
\end{equation}
Appearing on the left hand side of \eqref{eq-MAIN} is a so-called \textit{fractional} GJMS type operator to be described later and
\[
\Q = \gjms (1)
\] 
is constant. In terms of the Laplace--Beltrami operator $\Delta_{\S^n}$ on $\S^n$, the fractional operator $\gjms$ is given as follows
\begin{equation}\label{eq-GJMS}
\gjms = \frac{\Gamma(B+1/2+s)}{\Gamma(B+1/2-s)} \quad\text{with } B=\sqrt{-\Delta_{\S^n}+\frac{(n-1)^2}4}.
\end{equation}
Note that the action of the operators $B$ and $\gjms$ in any basis of spherical harmonics (and on spherical harmonics of degree $l \in \N_0 = \N\cup \{0\}$) is diagonal. Precisely, on spherical harmonics of degree $l\in \N_0$, the operator $B$ acts by multiplication with $l+ (n-1)/2$ and therefore the operator $\gjms$ acts by multiplication with
\begin{equation}\label{Alpha}
\alpha_{2s,n}(l)= \frac{\Gamma(l+n/2+s)}{\Gamma(l+n/2-s)}.
\end{equation}
In particular, there holds
\[
Q_n^{2s} = \gjms (1) = \frac{\Gamma (n/2 + s)}{\Gamma (n/2 - s )} > 0.
\]
Hence, the operator $\gjms$ can be thought of as $(-\Delta_{\S^n})^s$ perturbed by lower order terms. For integer $s$, this operator is related to the GJMS operators in conformal geometry; see \cite{GJMS, Beckner, FKT}.

In this paper, our range of $s$ is $s \in (0, n/2)$, which implies that the denominator in \eqref{Alpha} has no pole. Hence, it follows from \cite[page 3]{FKT} that
\begin{equation}\label{ProjectGJMS}
\int_{\S^n} v \gjms (v) d\mu_{g_{\S^n}} = \sum_{l\in \N_0} \alpha_{2s,n}(l) \|P_l v\|_2^2 \quad \text{for all } v \in H^s(\S^n),
\end{equation}
where $P_l$ is the projection onto spherical harmonics of degree $l \in \N_0$. Let us remark that, when $s$ is an integer and by directly verifying from \eqref{eq-GJMS} we have the following expression 
\begin{equation}\label{GJMS-integer}
\mathbf P_n^{2s} = \prod\limits_{k =1}^s { \Big( -\Delta_{\S^n} + \Big( \frac n2 - k \Big) \Big( \frac n2 + k - 1 \Big) \Big)}.
\end{equation}
In the literature, the operator $\gjms$ in \eqref{GJMS-integer} includes several important ones. For example, in the case $s=1$, the operator $\gjms$ in \eqref{GJMS-integer} becomes
	\[
	\mathbf P_2 = -\Delta_{\S^n} + \frac{n(n-2)}4,
	\]
	which is the well-known conformal Laplace operator on $\S^n$. This plays the central role in the Yamabe problem as well as the prescribed scalar curvature problem on $\S^n$. In the case $s=2$, the operator $\gjms$ in \eqref{GJMS-integer} becomes
	\[
	\mathbf P_4 =\Big( -\Delta_{\S^n} + \frac{n(n-2)}4 \Big)\Big( -\Delta_{\S^n} + \frac{(n-2)(n-4)}4 \Big),
	\]
	which is the well-known Paneitz operator on $\S^n$. This operator plays the central role the prescribed $Q$-curvature problem on $\S^n$. Equation \eqref{eq-MAIN} can be thought of as the perturbation of the following critical equation
	\[
	\gjms (v) = v^\frac{n+2s}{n-2s}
\quad \text{on } \S^n.
	\]
This equation belongs to a wider class of critical equations, known as the prescribing $Q$-curvature equation for (fractional) GJMS operators, whose the right hand side is $Q v^\frac{n+2s}{n-2s}$ for some given function $Q$ on $\S^n$. Thus, in some sense, by considering \eqref{eq-MAIN} we are interested in subcritical equation for constant $Q$-curvature equations for GJMS operators. 

Our motivation of working on the equation \eqref{eq-MAIN} traces back to a recent F. Hang and P. Yang; see \cite{HangYang}. In this work, a perturbation approach was used. In our language, this leads to the following higher-order equation
	\begin{equation}\label{eq-MAIN-HY}
\mathbf P_n^{2s} (v) = Q_n^{2s} (\ve v+v^{-\alpha}) \quad \text{on } \S^n
	\end{equation}
with odd $n<2s$ and $s$ is an integer. It was proved in \cite{Zhang} for the case $s=2$ and in \cite{HyderNgo} for the general case $s \geq 2$ that if $\ve$ is sufficiently small, then any positive, smooth solution $v$ to \eqref{eq-MAIN-HY} must be constant. Naturally, one wishes to consider the counter-part of \eqref{eq-MAIN-HY}, namely the equation \eqref{eq-MAIN}. However, there are also reasons why we are interested in the equation \eqref{eq-MAIN}. For example, the above equation is in much the same way as the higher-order Lane--Emden equation in $\R^n$, which is 
	\begin{equation*}\label{eq-LE}
		(-\Delta)^s u = u^\alpha \quad \text{in } \R^n;
	\end{equation*}
see \cite{WeiXu, NNPY} and the references therein, to name a few, for further discussion.

Back to the higher-order equation \eqref{eq-MAIN} on $\S^n$, as far as we know, \cite{CLS} is the first work studying Liouville type results. Among other results, it is proved in \cite{CLS} that any smooth, positive solution $v$ to \eqref{eq-MAIN} with $s$ integer, $1\leq \alpha < (n+2s)/(n-2s)$ must be constant. To obtain such a Liouville type result, the authors in \cite{CLS} follow a recent strategy often used to seek for Liouville type result. The strategy consists of two steps: first to transfer the differential equation \eqref{eq-MAIN} on $\S^n$ into a corresponding integral equation in $\R^n$, see \eqref{eq-MAIN-Rn-integral} below, then apply method of moving spheres/planes to classify solutions to the integral equation in $\R^n$; see also \cite{Zhang} and \cite{HyderNgo}. However, the most interesting case $\alpha = (n+2s)/(n-2s)$, namely the critical case, is left in \cite{CLS}. This motivates us.

In this work, to tackle \eqref{eq-MAIN}, we adopt the strategy used in \cite{Zhang} and in \cite{HyderNgo}. As usual, the first step in this strategy is to transfer \eqref{eq-MAIN} on $\S^n$ to the integral equation \eqref{eq-MAIN-Rn-integral} in $\R^n$ by means of the stereographic projection $\pi_N$ to be described later. This step is more or less well-known, but for completeness, we shall briefly sketch the derivation. Throughout the paper, to avoid any confusion, we use $v$ to denote a function in $\S^n$ while we use $u$ to denote a function in $\R^n$. Let us denote by $\pi_N : \S^n \to \R^n$ the stereographic projection from the north pole $N = (1,0,..,0) \in \S^n$. 
 It is well-known that
\[
(\pi_N^{-1})^* (g_{\S^n}) = \big( \frac 2{1+|x|^2} \big)^2 dx^2
\]
and making use of \cite[Lemma 4]{FKT} gives
\[
 (-\Delta)^{\lfloor s \rfloor- s } \Big( \big( \frac 2{1+|x|^2} \big)^{\frac {n+2s}2} \gjms (v) \circ \pi_N^{-1} \Big) =  (-\Delta)^{\lfloor s \rfloor} \Big( \big( \frac 2{1+|x|^2} \big)^{\frac {n-2s}2} v \circ \pi_N^{-1} \Big)
\]
in $\R^n$. (Here $\lfloor \cdot \rfloor$ is the usual floor function.) Therefore, if we let 
	\begin{equation}\label{eq-u=v}
		u = \big( \frac 2{1+|x|^2} \big)^{\frac {n-2s}2} (v \circ \pi_N^{-1}),
	\end{equation}
	then it follows from \eqref{eq-MAIN} that
	\begin{equation}\label{eq-MAIN-Rn}	
	\begin{aligned}
		(-\Delta)^{\lfloor s \rfloor} u &=   (-\Delta)^{\lfloor s \rfloor -s} \Big( \Q  \Big[ \varepsilon \big(\frac{2}{1+|x|^2}\big)^{2s} u+
		\big( \frac 2{1+|x|^2} \big)^{\frac {n+2s}2} \big( v \circ \pi_N^{-1} \big)^\alpha \Big] \Big) \\
		&=   (-\Delta)^{\lfloor s \rfloor -s} \Big( \Q  \Big[ \varepsilon \big(\frac{2}{1+|x|^2}\big)^{2s} u+ \big( \frac 2{1+|x|^2} \big)^{\frac {n+2s}2 - \alpha \frac {n-2s}2} u^\alpha  \Big] \Big)
	\end{aligned}
	\end{equation}	
in $\R^n$. For simplicity, we denote
\begin{equation}\label{eq-F}
F_{\ve, u} (y) = \varepsilon \big(\frac{2}{1+|y|^2}\big)^{2s} u(y)+\big( \frac 2{1+|y|^2} \big)^{\frac {n+2s}2 - \alpha \frac {n-2s}2} u^\alpha(y)
	\end{equation}
and	
	\[
	\sigma = \frac {n+2s}2 - \alpha \frac {n-2s}2
	\]
which is non-negative due to the conditions $0 < \alpha \leq (n+2s)/(n-2s)$ and $n>2s$. Thus, we have just shown that $u$ solves
	\begin{equation}\label{eq-MAIN-Rn}
		(-\Delta)^{\lfloor s \rfloor} u= (-\Delta)^{\lfloor s \rfloor -s}  ( \Q F_{\ve, u} ) \quad\text{in } \R^n.
	\end{equation}
However, our job is not over yet. In fact, we still need to transfer the differential equation \eqref{eq-MAIN-Rn} to some integral equation. To be able to describe the procedure, a notation is needed. For $0 < \alpha < n$ we denote the constant $C(\alpha)$ as follows
	\begin{equation}\label{eq-CAlpha}
		C(\alpha) := \Gamma \big( \frac {n-\alpha}2 \big) \Big[ 2^\alpha \pi^{n/2} \Gamma \big(\frac \alpha 2 \big) \Big]^{-1} .
	\end{equation}
Under the condition $n>2s$ it is clear from \eqref{eq-CAlpha} that $C(2s)>0$. Furthermore, by a direct computation we can verify that $C(2s)|x|^{2s-n}$ is the fundamental solution of the (fractional) polyharmonic equation $(-\Delta)^{s} u = 0$ in $\R^n \setminus \{0\}$. Indeed, using the definition of fractional Laplacian via the Fourier transform, one has for any $\theta >0$ that 
\[
\mathcal F( (-\Delta)^{s} |x|^{-\theta})(\xi) = \frac{(2\pi)^n}{C(\theta) }|\xi|^{\theta+2s-n}.
\]
By taking $\theta =n-2s>0$ and the inverse Fourier transform, we obtain that 
\[
(-\Delta)^{s} \big( |x|^{2s-n} \big) = \frac1{C(2s)}\delta_0,
\]
where $\delta_0$ is the Dirac function at the origin. See also \cite[chapter 2]{CLM}. With the higher power fractional Laplacian $(-\Delta)^{s}$ being understood as 
\[
(-\Delta)^{s} = (-\Delta)^{s - \lfloor s \rfloor} \circ (-\Delta)^{\lfloor s \rfloor}, 
\] 
heuristically speaking, equation \eqref{eq-MAIN-Rn} can be seen as 
\[
(-\Delta )^s u =\Q F_{\ve, u}.
\] 
(See the discussion in \cite[page 8]{FKT} why we cannot directly claim this, but this is just as good for us to prove \eqref{eq-MAIN-Rn-integral} below.) Having the fundamental solution of the polyharmonic equation, one expects that the corresponding integral equation for solution $u$ to \eqref{eq-MAIN-Rn} in $\R^n$ should be
\begin{equation}\label{eq-MAIN-Rn-integral}
u(x) = \gamma_{2s,n} \int_{\R^n} \frac 1{|x-y|^{n-2s}}F_{\ve, u}(y) dy \quad \text{in } \R^n
\end{equation}
with
\[
\gamma_{2s, n} := C(2s) \Q >0.
\]	
In the first result of the paper we show that this is indeed true. We turn this into a theorem which is stated as follows.
	
	\begin{theorem}\label{thm-DE->IE}
		If $v$ is a non-negative, non-trivial, smooth solution to \eqref{eq-MAIN} on $\S^n$, then via the stereographic projection $\pi_N$ from the north pole $N$, the function $u$ defined by \eqref{eq-u=v}, namely
		\[
		u(x) = \big( \frac 2{1+|x|^2} \big)^{\frac {n-2s}2} \big( v \circ \pi_N^{-1} \big) (x),
		\]
		solves \eqref{eq-MAIN-Rn-integral} in $\R^n$. Moreover, $u$ is strictly positive, and hence $v$ is also strictly positive.
	\end{theorem}
	
As described above, Theorem \ref{thm-DE->IE} is now a standard, and there are several ways to prove it. If $s$ is an integer, a typical way to prove is to make use of the super polyharmonic property for solutions $u$ to \eqref{eq-MAIN-Rn}. Using this important property, combined with the use of Green function for Laplace operator on unit ball, one can derive \eqref{eq-MAIN-Rn-integral}. In fact, this is what the authors in \cite{CLS} used to obtain \eqref{eq-MAIN-Rn-integral}. Unfortunately, to be able to establish the super polyharmonic property for solutions, the condition $\alpha \geq 1$ plays a crucial role. This explains why the authors in \cite{CLS} cannot touch the range $0<\alpha<1$. The motivation of writing this paper also starts from this point. There is another way introduced in \cite{CAM08} to get \eqref{eq-MAIN-Rn-integral} without using the super polyharmonic property for solutions $u$ to \eqref{eq-MAIN-Rn}. Instead, the method introduced in \cite{CAM08} exploits the advantages of distributional solutions to \eqref{eq-MAIN-Rn}. In practice, a smooth solution is often a distributional one, hence yielding \eqref{eq-MAIN-Rn-integral} instantaneously. 
	
	In this paper, by making use of the advantages of \eqref{eq-u=v}, we propose another proof, 
which allows us, at the same time, to treat the case $0<\alpha<1$ being left in \cite{CLS} and the case $s$ non-integer. Up to this point, we know that any smooth, positive solution $v$ to \eqref{eq-MAIN} gives rise to a smooth, positive solution $u$ to \eqref{eq-MAIN-Rn-integral}. As the second step of the strategy, one could obtain a Liouville type result from the integral equation \eqref{eq-MAIN-Rn-integral}. Let us now state the second result of this paper.
	
	\begin{theorem}\label{thm-Liouville}
Let $s \geq 1$ and $n>2s$. Then under one of the following conditions
\begin{enumerate}
 \item either $\ve >0$ and $0 < \alpha \leq (n+2s)/(n-2s)$ 
 \item or $\ve = 0$ and $0 < \alpha < (n+2s)/(n-2s)$ 
\end{enumerate}
any non-negative, non-trivial, smooth solution $v$ to \eqref{eq-MAIN} on $\S^n$, if exists, must be constant.
	\end{theorem}

It is worth emphasizing that Theorem \ref{thm-Liouville} can be though of a prior result. To be precise, it only tells us that if a solution $v$ to \eqref{eq-MAIN} exists, then it must be constant. By a simple calculation, it can be verified that if $\alpha \ne 1$, then the non-zero constant solution $v$ to \eqref{eq-MAIN} is $(1-\ve)^{1/(\alpha-1)}$. However, in the case $\alpha=1$, there is no non-zero constant solution $v$ to \eqref{eq-MAIN}. A similar situation occurs when $\ve \geq 1$, as the corollary below, just a direct consequence of Theorem \ref{thm-Liouville}, shows.

\begin{corollary}\label{cor-Liouville}
Let $n>2s \geq 2$ and $\ve \geq 1$. Then, equation \eqref{eq-MAIN} on $\S^n$ admits no non-negative, non-trivial, smooth solution.
\end{corollary}	
	
In case $s$ is an integer, Corollary \ref{cor-Liouville} can be proved directly by integrating both sides of \eqref{eq-MAIN}. Indeed, as $n>2s >0$ and $\Q > 0$ by integrating we arrive at
\[
 ( 1- \ve ) \int_{\S^n} v d\mu_{\S^n} = \int_{\S^n} v^\alpha d\mu_{\S^n}.
\]
This immediately tells us that $\ve < 1$. Theorem \ref{thm-Liouville} deserves further comments as follows. 
	
\begin{itemize}
 \item It is not difficult to see that the strict inequality $\alpha < (n+2s)/(n-2s)$ for the case $\ve = 0$ in Theorem \ref{thm-Liouville} is \textit{optimal}. In other words, the above Liouville type result does not hold when $\ve =0$ and $\alpha = (n+2s)/(n-2s)$. This is because the corresponding equation \eqref{eq-MAIN-Rn} becomes
\[
(-\Delta)^{s} u= u^\frac{n+2s}{n-2s} \quad \text{in } \R^n
\]
and in this particular case we know that (see e.g. \cite[Theorem 2]{CLZ}) there are many positive, smooth solution of the form
	\[
	u(x) = a \Big(\frac 2{b^2+|x-x_0|^2} \Big)^\frac{n-2s}{2}.
	\]
	for some $x_0 \in \R^n$ and some constants $a, b>0$. This and \eqref{eq-u=v} show that there are many non-constant solutions $v$ to \eqref{eq-MAIN}. Therefore, our Theorem \ref{thm-Liouville} shows how does the linear perturbation effect the critical equation.
	
 \item An immediate consequence of Theorem \ref{thm-Liouville} is that the function $u$ defined by \eqref{eq-u=v} is of the form
	\[
	u(x) = c \Big(\frac 2{1+|x|^2} \Big)^{\frac {n-2s}2}
	\] 
	for some positive constant $c>0$. In the course of the proof of Theorem \ref{thm-Liouville}, we crucially make use of the structure of $\S^n$ as well as the relation \eqref{eq-u=v}. For example, the relation \eqref{eq-u=v} provides us the exact asymptotic behavior of $u$ at infinity. Therefore, by relaxing the relation \eqref{eq-u=v}, it is natural to ask if any smooth, positive solution $u$ to \eqref{eq-MAIN-Rn} in $\R^n$, namely the equation
\[
(-\Delta)^{s} u= \varepsilon \big(\frac{2}{1+|x|^2}\big)^{2s } u+ \big( \frac 2{1+|x|^2} \big)^{\sigma} u^\alpha,
\]
at least in the case $s$ is an integer and $\alpha \ne 1$, is of the above form, up to translations and dilations. It should be noted that in the special case $\ve = 0$, the equation \eqref{eq-MAIN-Rn} in $\R^n$ is very similar to the Matukuma equation; see \cite{Mat} and \cite[Eq. (4.21)]{BFH86}. Therefore, one can ask many similar questions, for example, is any solution to \eqref{eq-MAIN-Rn} radially symmetric? See \cite{YiLi}. Toward an answer for this question, we expect that the method developed in \cite{CAM08} and \cite{NY} could be useful.
\end{itemize}
	
Let us sketch our approach to prove Theorem \ref{thm-Liouville}. With help of Theorem \ref{thm-DE->IE}, we mainly work on the integral equation \eqref{eq-MAIN-Rn-integral} and our aim is to prove that any positive, smooth solution $u$ to \eqref{eq-MAIN-Rn-integral} must be radially symmetric with respect to the origin; see Lemma \ref{lem-MS-stop}. Unlike the similar work \cite{CLS} where the method of moving planes in integral form was used, we use the method of moving spheres. Once we have the symmetry of $u$, by using \eqref{eq-u=v} we are able to conclude that $v$ must be constant. This completes the proof of Theorem \ref{thm-Liouville}.
	
Finally, to illustrate our finding on a Liouville type result for solutions to \eqref{eq-MAIN}, we revisit the sharp subcritical/critical Sobolev inequality (see  \cite{Beckner}, see also \cite{CLS}) for $\gjms$ on $\S^n$, restricted to non-negative functions. 

\begin{theorem}\label{thm-Sobolev}
Let $n > 2s \geq 2$ and $0<\alpha \leq (n+2s)/(n-2s)$. Then, for any non-negative $v \in H^s (\S^n)$, we have the following sharp Sobolev inequality
\begin{equation}\label{eq-sSobolev}
\int_{\S^n} v \gjms (v) d\mu_{g_{\S^n}}
\geq \frac{\Gamma (n/2 + s)}{\Gamma (n/2 - s )}| \S^n|^\frac{\alpha-1}{\alpha+1}
\Big( \int_{\S^n} v^{\alpha+1} d\mu_{g_{\S^n}} \Big)^\frac{2}{\alpha+1}.
\end{equation} 
Moreover, equality occurs if $v$ is any positive constant.
\end{theorem}

We prove Theorem \ref{thm-Sobolev} in section \ref{sec-Sobolev}. We shall explain in the beginning of section \ref{sec-Sobolev} that it suffices to consider only the critical case $\alpha= (n+2s)/(n-2s)$. In this critical case, the strategy is as follows. First we transfer a perturbation of \eqref{eq-sSobolev} into a minimizing problem 
\begin{equation}\label{eq-SS-O}
\mathcal S_\ve = \inf_{v \in \mathcal W }\int_{\S^n} \big[ v \gjms (v) - \ve \Q v^2 \big] d\mu_{g_{\S^n}}
\end{equation}
within the set
\[
\mathcal W =\Big\{0\leq v \in H^s (\S^n) : \int_{\S^n} v^\frac{2n}{n-2s} d\mu_{g_{\S^n}} = 1 \Big\}.
\]
Second, to look for an optimizer $v_\ve$, we need to employ a concentration-compactness principle to gain a strong convergence; see Lemma \ref{lem-SSexistence}. Precisely, $v_\ve$ solves
\begin{equation}\label{EqVe}
\gjms (v_\ve) = \ve \Q v_\ve +\mathcal S_\ve v_\ve^\frac{n+2s}{n-2s}  \quad \text{on $\S^n$}.
\end{equation}
Without the linear perturbation in the right-hand side of \eqref{EqVe},  it follows from the dual argument of  \cite{CLS} that any solution of \eqref{EqVe} as $\ve=0$  is always non-negative. However, in this paper, due to the presence of the linear perturbation, we cannot obtain  the same result. As consequence,  we have to restrict our constraint $\mathcal W$ to non-negative functions in $H^s(\S^n)$ and study the critical Sobolev inequality  in this class. Once we have a non-negative solution $v_\ve$ of \eqref{EqVe}, we can apply the Liouville type result. Through a limiting process to arrive at \eqref{eq-sSobolev}, we compute the  sharp constant.

Before closing this section, we briefly mention the organization of the paper. 
	
\tableofcontents
	 	
	
	\section{From differential equations to integral equations: proof of Theorem \ref{thm-DE->IE}}
	\label{sec-DE->IE}
	
	This section is devoted to a proof of Theorem \ref{thm-DE->IE}, namely we show that if $v$ is a non-negative, non-trivial, smooth solution to \eqref{eq-MAIN} on $\S^n$, then the corresponding function $u$ defined by \eqref{eq-u=v} solves \eqref{eq-MAIN-Rn-integral} in $\R^n$. However, it is clear from the previous section that it suffices to show that if $u$ solves \eqref{eq-MAIN-Rn}, then $u$ also solves \eqref{eq-MAIN-Rn-integral}. To this purpose, we let $\widetilde u$ be the right hand side of \eqref{eq-MAIN-Rn-integral}, namely
\[
\widetilde u(x) = \gamma_{2s,n} \int_{\R^n} \frac 1{|x-y|^{n-2s}} F_{\ve, u} (y) dy\quad \text{in } \R^n
\]
with $F_{\ve, u} (y)$ given in \eqref{eq-F} and $\gamma_{2s, n} = C(2s) \Q >0$ as above. We have to prove that 
\[
u \equiv \widetilde u \quad \text{everywhere in } \R^n.
\] 
In the first step, we show that 
\begin{equation}\label{AsymptoticU-U}
u(x) - \widetilde u(x) = O \big(\frac 1{|x|^{n-2s}} \big)
\end{equation}
at infinity. From the definition of $u$, see \eqref{eq-u=v}, namely
	\[
	u(x) = \big( \frac 2{1+|x|^2} \big)^{\frac {n-2s}2} \big( v \circ \pi_N^{-1} \big) (x),
	\]
and the smoothness of $v$ on $\S^n$, it is clear that $ u(x) = O( |x|^{2s-n} )$ near infinity. Next, we shall see that the function $\widetilde u$ also has the decay as that of the function $|x|^{2s-n}$ at infinity. This is done if one can show that
\[
\int_{\R^n} \frac 1{|x-y|^{n-2s}} F_{\ve, u} (y) dy = O \big(\frac 1{|x|^{n-2s}} \big).
\]
To do so, first keep in mind that
\[
\big( \frac 2{1+|y|^2} \big)^{\sigma} u^\alpha(y) = \big( \frac 2{1+|y|^2} \big)^{\frac {n+2s}2} (v\circ \pi_N^{-1})^\alpha (y) \sim \frac 1{|y|^{n+2s}}
\]
and that
\[
\big( \frac 2{1+|y|^2} \big)^{2s} u(y) = \big( \frac 2{1+|y|^2} \big)^{\frac {n+2s}2} (v\circ \pi_N^{-1})(y) \sim \frac 1{|y|^{n+2s}}
\]	
at infinity. Hence, there is some $C>0$ such that
\begin{equation}\label{eq-F<=}
|F_{\ve, u} (y)|	= \ve \big( \frac 2{1+|y|^2} \big)^{2s} u (y) + \big( \frac 2{1+|y|^2} \big)^{\sigma} u^\alpha(y) < \frac C{|y|^{n+2s}}
\end{equation}
everywhere in $\R^n$. Due to the smoothness, it suffices to consider $|x| \geq 2$. We decompose $\R^n$ as follows
\[
\R^n = B_{|x|/2}(x) \cup \big[ \R^n \setminus B_{|x|/2}(x) \big]
\]
Obviously, any $y \in \R^n \setminus B_{|x|/2}(x) $ enjoys the estimate $2|x-y| \geq |x|$. This gives
\begin{align*}
\int_{\R^n \setminus B_{|x|/2}(x) } & \frac {|x|^{n-2s}}{|x-y|^{n-2s}} F_{\ve, u} (y) dy\lesssim \int_{\R^n} \big( \frac 2{1+|y|^2} \big)^{\sigma} u^\alpha(y) dy < +\infty.
\end{align*}
Now for $y \in B_{|x|/2}(x)$ we can estimate $|y| \geq |x| - |x-y| \geq |x|/2 \geq 1$. Hence, together with \eqref{eq-F<=}, we obtain
\begin{align*}
\int_{B_{|x|/2}(x) } \frac {|x|^{n-2s}}{|x-y|^{n-2s}} F_{\ve, u} (y) dy 
&< C |x|^{n-2s}\int_{B_{|x|/2}(x) } \frac 1{|x-y|^{n-2s}} \frac 1{|y|^{n+2s}}dy \\
&\lesssim \frac 1{ |x|^{4s}} \int_{B_{|x|/2}(x) } \frac {dy}{|x-y|^{n-2s}} \\
& \lesssim \frac 1{ |x|^{2s}} \lesssim 1 .
\end{align*}
Putting the above computations together gives $\widetilde u(x) = O( |x|^{2s-n} )$. Thus, we have shown that \eqref{AsymptoticU-U} holds. 

In the next step, we prove that 
\begin{equation}\label{eqDeltaU-U}
(-\Delta)^{\lfloor s \rfloor} \big[ u(x) - \widetilde u(x) \big] =0\quad \text{in } \R^n.	
\end{equation}
Note, for $|x|$ large, that $|F_{\ve, u}(x)| \lesssim (1+|x|^2)^{(-2s-n)/2}$, see \eqref{eq-F<=}. Hence, we can use the following identity from \cite[page 8]{FKT} (see the last line) to obtain that
\[
(-\Delta)^{\lfloor s \rfloor} \int_{\R^n}|x-y|^{2s-n}F_{\ve, u}(y) dy= \frac 1{C(2s)} (-\Delta)^{\lfloor s \rfloor -s} \big( F_{\ve, u}(x) \big). 
\]
It yields
\[
(-\Delta)^{\lfloor s \rfloor} \widetilde u =  \Q (-\Delta)^{\lfloor s \rfloor - s} F_{\ve, u} \quad \text{in } \R^n.	
\]
Together with \eqref{eq-MAIN-Rn}, we arrive at \eqref{eqDeltaU-U}. 
	
In view of \eqref{eqDeltaU-U}, the function $w=u-\widetilde u$ is poly-harmonic in $\R^n$. Hence, we make use of the Liouville theorem for polyharmonic functions, see \cite[Theorem 5]{Martinazzi}, to get that $w$ is a polynomial. (See \cite[Lemma 5.7]{AGHW} for a similar result in the fractional setting.) However, as $w$ vanishes at infinity because of \eqref{AsymptoticU-U}, it must be identically zero, i.e. $u\equiv \widetilde u$ in $\R^n$. 



	
Finally, we prove that $u$ must be strictly positive. If this is the case, then by \eqref{eq-u=v} we conclude that $v$ is also strictly positive. To this end, we assume that there is some $x_0 \in \R^n$ such that $u(x_0)=0$. Using \eqref{eq-MAIN-Rn-integral} we should have
\[
		0 = u(x_0) = \gamma_{2s,n} \int_{\R^n} \frac 1{|x_0-y|^{n-2s}}F_{\ve, u} (y) dy.
\]		
Keep in mind that $\gamma_{2s,n} >0$. The only possibility is that $u \equiv 0$ in $\R^n$. Hence $u$ is trivial. This completes our proof.
	
	
	\section{Classification of solution via the method of moving spheres: proof of Theorem \ref{thm-Liouville}}
	\label{sec-IE-classification}
	
	This section is devoted to the proof of Theorem \ref{thm-Liouville}, namely, we shall show that any non-negative, non-trivial, smooth solution $v$ to \eqref{eq-MAIN} is necessarily constant. Thanks to Theorem \ref{thm-DE->IE}, we can assume at the beginning that $v$ is positive everywhere on $\S^n$. 
	
To prove Theorem \ref{thm-Liouville}, we make use of the method of moving spheres to prove the radial symmetry of the solution $u$ to \eqref{eq-MAIN-Rn-integral} in the spirit of \cite{Li} and \cite{JLX08}, see also \cite{HuiYang}, with some modifications due to the presence of the weight $(1+|y|^2)^{-\sigma}$. It is worth noting that in the case $n >2s$ we still can make use of the method of moving planes, resulting in the radial symmetry of solutions $u$ to \eqref{eq-MAIN-Rn} with respect to the origin. This is the approach used in \cite{CLS}. However, it is not clear for us how to show that $u$ must be of the form $(1+|x|^2)^{(2s-n)/2}$. 
	
In the next paragraph, we recall basics of the inversion and the Kelvin transform necessarily for the method of moving spheres.
	
	Given the parameter $\lambda>0$ we denote by $\xi ^ {x,\lambda}$ the inversion of $\xi \in \R^n$ via the sphere $\partial B_\lambda(x) $ with radius $\lambda$ and center at $x \in \R ^n$. 	
\begin{figure}[H]
\begin{tikzpicture}[scale=0.75]
\draw[-] (0,0) circle (2.5);
\draw[-] (0, 0) node[anchor = east] {$x$} 
-- (1.5,1) node[anchor =south] {$\xi$} 
-- (3,2) node[anchor =west] {$\xi^{x,\lambda} =x+ \frac{\lambda^2 (\xi-x)}{|\xi- x|^2}$} 
-- (0.75,-0.25);
\draw[dashed] (0,0) -- (0.3, 2.485) -- (3,2);
\draw[-] (0, 0) -- (0.75,-0.25) node[anchor =north] {$z^{x,\lambda}$} 
-- (6,-2) node[anchor =west] {$z = x+\frac{\lambda^2 (z^{x,\lambda} - x)}{|z^{x,\lambda}-x|^2}$}
-- (1.5,1) ;
\draw[dashed] (0,0) -- (0.2, -2.485) -- (6,-2);
\node [anchor =north] at (-2,-2) {$\partial B_\lambda (x)$};
\end{tikzpicture}
\caption[]{Inversion in the method of moving spheres.}\label{fig-MS}
\end{figure}
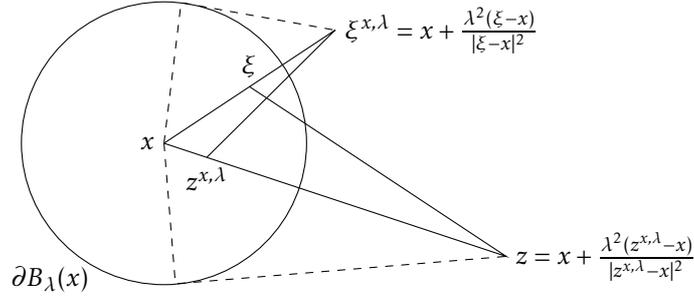

Precisely, $\xi ^ {x,\lambda}$ is given as follows
	\[
	\xi ^ {x,\lambda}= x+ \frac{\lambda^2(\xi-x)}{|\xi-x|^2} \quad \text{for } \xi \neq x. 
	\] 
	It is clear that
	\[
	d(\xi ^
	{x,\lambda}) = \big(\frac{\lambda }{|\xi-x|} \big)^{2n} d\xi
	\] 
and
\begin{equation}\label{eq-|Kelvin|}
|z-x| |\xi-x||\xi^{x,\lambda} - z^{x,\lambda}| = \lambda^2 |\xi-z|.
\end{equation}
	We also denote by $u_{x,\lambda}$ the Kelvin transform of $u$ via the sphere $\partial B_\lambda $, namely
	\[
	u_{x,\lambda} (\xi) = \big( \frac {\lambda}{|\xi-x|} \big)^{n-2s} u (\xi ^ {x,\lambda}) = \big( \frac {\lambda}{|\xi-x|} \big)^{n-2s} u \big( x+ \frac{\lambda^2(\xi-x)}{|\xi-x|^2} \big) .
	\]
	Clearly, the Kelvin transform $u_{x,\lambda}$ is defined in $\R^n \setminus \{ 0 \}$ and
\begin{equation*}\label{eq-KelvinKelvin}
u_{x,\lambda} (\xi^ {x,\lambda}) = \big( \frac {\lambda}{|\xi^ {x,\lambda}-x|} \big)^{n-2s} u (\xi)
= \big( \frac {|\xi-x|}{\lambda} \big)^{n-2s} u (\xi).
\end{equation*}
	The basic idea of the method of moving spheres is to compare $u$ and $u_{x,\lambda}$ pointwise starting from small $\lambda >0$. More precisely, we shall show for small $\lambda > 0$ that
\[
u(y) \geq u_{x,\lambda} (y)
\] 
for any $y$ satisfying $|y-x| \geq \lambda> 0$; see Lemma \ref{lem-MS-small}. Then we increase $\lambda$ toward $\sqrt{1+|x|^2}$ until the above inequality does not hold; see \eqref{eq-lambda}. We shall show in Lemma \ref{LeM=002} that $\lambda \geq |x|$, and at the moment when $\lambda$ passes $|x|$, we are able to capture further information on $u$; see Lemma \ref{lem-MS-stop}.
	
	To be able to compare $u$ and $u_{x,\lambda}$, our proof starts with the following simple observation for $u_{x,\lambda}$.
	
	\begin{lemma}\label{lem-ul}
		There holds
		\begin{align}\label{eq-u_{x,lambda}}
			u_{x,\lambda}(\xi)=\gamma_{2s,n} \int_{ \R ^n} \frac{1}{|\xi-z|^{n-2s}} \; \left[
			\begin{aligned}
			& \varepsilon\big(\frac{2 \lambda^2}{1+|z^{x, \lambda}|^2}\big)^{2s}\; \frac{1}{|z-x|^{4s}}u_{x, \lambda}(z) \\
			&+ \big( \frac {2 \lambda^2}{1+|z^{x, \lambda}|^2} \big)^\sigma \frac{1}{|z-x|^{2 \sigma}} \; u_{x,\lambda}^\alpha (z ) 
			\end{aligned} 
			\right] dz
		\end{align}
		in $\R^n$.
	\end{lemma}
	
	\begin{proof}
		This is straightforward. For completeness, we include its proof here. With $y=z^{x,\lambda}$, the relation \eqref{eq-|Kelvin|}, and the identity $2\sigma = n+2s - \alpha (n-2s)$ in hand, we easily get
		\begin{align*}
u_{x,\lambda} (\xi) & = \big( \frac{\lambda}{|\xi-x|} \big)^{n-2s} u(\xi^{x,\lambda})\\
			&= \gamma_{2s,n} \big( \frac{\lambda}{|\xi-x|} \big)^{n-2s} \int_{ \R ^n} 
\left[
			\begin{aligned}
			&\varepsilon \big(\frac{2}{1+|z^{x, \lambda}|^2}\big)^{2s} u(z^{x, \lambda}) \\
			&+\big( \frac 2{1+|z^{x, \lambda}|^2} \big)^\sigma u^\alpha(z^{x,\lambda})
			\end{aligned} 
			\right] 		\frac {d(z^{x, \lambda})}{|\xi^{x,\lambda} -z^{x,\lambda}|^{n-2s}}.	
		\end{align*}
Note that
\[
u (z^{x,\lambda}) = ( \frac {|z-x|}\lambda )^{n-2s} u_{x,\lambda} (z),
\quad
u^\alpha (z^{x,\lambda}) = \big( \frac {|z-x|}\lambda \big)^{(n-2s)\alpha} u_{x,\lambda}^\alpha (z) ,
\]
and
\begin{equation}\label{eq-dz=dz}
d(z^{x,\lambda}) = \big(\frac{\lambda }{|z-x|} \big)^{2n} dz.
\end{equation}
			Therefore, we obtain
			\begin{align*}
		u_{x,\lambda} (\xi)	=\gamma_{2s,n} \int_{ \R ^n} & \frac 1{ (|\xi-x||\xi^{x,\lambda} -z^{x,\lambda}|)^{n-2s}}\\
& \times \left[
			\begin{aligned}
&\varepsilon \lambda^{n-2s} \big(\frac{|z-x|}{\lambda}\big)^{n-2s } \big(\frac{\lambda}{|z-x|}\big)^{2n}\big( \frac {2}{1+|z^{x, \lambda}|^2} \big)^{2s} u_{x, \lambda}(z)\\
		&+ \frac{\lambda^{n-2s} \lambda^{2n}}{\lambda^{\alpha (n-2s)}} 
			\; \big( \frac {2}{1+|z^{x, \lambda}|^2} \big)^\sigma\; \frac{|z-x|^{\alpha (n-2s)}}{|z-x|^{2n}} u_{x,\lambda}^\alpha (z )
\end{aligned} 
			\right] dz\\
			= \gamma_{2s,n} \int_{ \R ^n}& 
\left[
			\begin{aligned}
			&\varepsilon \frac{|z-x|^{n-2s}}{\lambda^{2(n-2s)} |\xi-z|^{n-2s}}.\; \frac{\lambda^{2n}}{|z-x|^{n+2s}}\; \big(\frac{2}{1+|z^{x, \lambda}|^2}\big)^{2s}\; u_{x, \lambda}(z)\\
			& + \frac{1}{|\xi-z|^{n-2s}} \big( \frac {2 \lambda^2}{1+|z^{x, \lambda}|^2} \big)^\sigma \frac{1}{|z-x|^{2 \sigma}} \; u_{x,\lambda}^\alpha (z )
\end{aligned} 
			\right] dz\\
			= \gamma_{2s,n} \int_{ \R ^n} & 
\left[
			\begin{aligned}
			& \varepsilon \frac{1}{|z-x|^{4s}} \frac{1}{|\xi -z|^{n-2s}} \big(\frac{2 \lambda^2}{1+|z^{x, \lambda}|^2}\big)^{2s}\; u_{x, \lambda}(z)\\
			& + \frac{1}{|\xi-z|^{n-2s}} \big( \frac {2 \lambda^2}{1+|z^{x, \lambda}|^2} \big)^\sigma \frac{1}{|z-x|^{2 \sigma}} u_{x,\lambda}^\alpha (z )
\end{aligned} 
			\right]	 dz.
		\end{align*}
		
From this one easily gets \eqref{eq-u_{x,lambda}} as claimed. The proof is complete.
	\end{proof}
	
	Next, we can estimate the difference $u - u_{x,\lambda}$. This step is necessary to run the method of moving spheres.
	
	\begin{lemma}\label{lem-U-U}
		We have
		\begin{align}\label{eq-u-u=FG}
			u (\xi) - u_{x,\lambda} (\xi) &= \gamma_{2s,n} \int_{|z-x| \geq \lambda} K(x, \lambda; \xi,z) G_{\ve,\lambda, u}(z) dz 
		\end{align}
		with
\[
G_{\ve,\lambda, u}(z) = \left[
			\begin{aligned}		
			& \varepsilon \big(\frac{2}{1+|z|^2}\big)^{2s} u(z) +\big( \frac {2}{1+|z|^2} \big)^\sigma \; u^\alpha (z) \\
			& - \varepsilon \big(\frac{2 \lambda^2}{1+|z^{x, \lambda}|^2}\big)^{2s} \frac{1}{|z-x|^{4s}}\; u_{x, \lambda}(z) \\
			& -\big( \frac {2 \lambda^2}{1+|z^{x,\lambda}|^2} \big)^\sigma \frac{1}{|z-x|^{2 \sigma}} \; u_{x, \lambda}^{\alpha}(z) 
\end{aligned} 
			\right]	
\]
and the kernel $K(x, \lambda; \xi,z)$ given by
\begin{equation}\label{eq-K}
		\begin{aligned}
K(x, \lambda; \xi,z)&= \frac 1{|\xi-z|^{n-2s}} - \big( \frac {\lambda}{|\xi-x|} \big)^{n-2s} \frac 1{|\xi^{x, \lambda}-z|^{n-2s}}\\
			& = \frac{1}{|\xi-z|^{n-2s}}- \frac 1{ |\xi-z^{x,\lambda}|^{n-2s}} \big( \frac {\lambda}{|z-x|} \big)^{n-2s}.
		\end{aligned}
\end{equation}		
Moreover, the kernel $K>0$ for any $|\xi-x|>\lambda$ and $|z-x| > \lambda$. 
	\end{lemma}
	
	\begin{proof}
		This is also straightforward. First, we estimate $u$ by using \eqref{eq-MAIN-Rn-integral}. To do so, we decompose the integral $\int_{ \R ^n}$ in \eqref{eq-MAIN-Rn-integral} into two parts $ \int_{\{z:|z-x| \geq \lambda\}} + \int_{\{z:|z-x| \leq \lambda\}}$ to get
\begin{align*}
u (\xi)=& \gamma_{2s,n} \int_{\{z:|z-x| \geq \lambda\}} \frac 1{|\xi-z|^{n-2s}} F_{\ve, u} (z) dz \\
&\quad+ \gamma_{2s,n} \int_{\{z^{x, \lambda}:|z^{x, \lambda}-x| \leq \lambda\}} \frac 1{|\xi-z^{x, \lambda}|^{n-2s}} F_{\ve, u} (z^{x, \lambda}) d(z^{x, \lambda}).
\end{align*}
For the integral $\int_{\{z^{x, \lambda}:|z^{x, \lambda}-x| \leq \lambda\}}$ we simply make use of \eqref{eq-dz=dz} and the change of the variable $z \mapsto z^{x, \lambda}$ to get 	
		\begin{align*}
			u (\xi)=& \gamma_{2s,n} \int_{|z-x| \geq \lambda} \frac 1{|\xi-z|^{n-2s}} F_{\ve, u} (z) dz \\
			&+ \gamma_{2s,n} \int_{|z-x| \geq \lambda} \frac 1{|\xi-z^{x, \lambda}|^{n-2s}} F_{\ve, u} (z^{x, \lambda}) \big(\frac{\lambda }{|z-x|} \big)^{2n} dz.
		\end{align*}
As
\[
u_{x,\lambda} (z) = \big( \frac {\lambda}{|z-x|} \big)^{n-2s} u (z^ {x,\lambda})
\] 
we further get
	\begin{align*}
u(\xi) =& \gamma_{2s,n}\int_{|z-x| \geq \lambda} \frac{1}{|\xi-z|^{n-2s}}\Big[ \varepsilon \big(\frac{2}{1+|z|^2}\big)^{2s} u(z) +\big( \frac {2}{1+|z|^2} \big)^\sigma \; u^\alpha (z ) \Big]\; dz \\
			&+\gamma_{2s,n} \int_{|z-x| \geq \lambda } \frac 1{ |\xi-z^{x,\lambda}|^{n-2s}}
\left[
			\begin{aligned}	
			& \varepsilon \big(\frac{2}{1+|z^{x, \lambda}|^2}\big)^{2s} \big(\frac{\lambda}{|z-x|}\big)^{n+2s} u_{x, \lambda}(z) \\
&+\big( \frac {2}{1+|z^{x,\lambda}|^2} \big)^\sigma \big(\frac{\lambda}{|z-x|}\big)^{2n-\alpha(n-2s)} u_{x, \lambda}^{\alpha}(z)
\end{aligned} 
			\right]				dz.
		\end{align*}
		In a same way, we obtain a similar decomposition of $u_{x,\lambda}$ by using \eqref{eq-u_{x,lambda}} in Lemma \ref{lem-ul}. Indeed, by letting
\[
F^\lambda_{\ve, u} (z) = \varepsilon\big(\frac{2 \lambda^2}{1+|z^{x, \lambda}|^2}\big)^{2s}\; \frac{1}{|z-x|^{4s}}u_{x, \lambda}(z) + \big( \frac {2 \lambda^2}{1+|z^{x, \lambda}|^2} \big)^\sigma \frac{1}{|z-x|^{2 \sigma}} \; u_{x,\lambda}^\alpha (z ) 
\]
we clearly have
\begin{align*}		
			u_{x,\lambda}(\xi)= & \gamma_{2s,n} \int_{|z-x|\geq \lambda} \frac{1}{|\xi-z|^{n-2s}} F^\lambda_{\ve, u} (z) dz\\
 & + \gamma_{2s,n} \int_{\{z^{x, \lambda} : |z^{x, \lambda}-x|\leq \lambda\}} \frac{1}{|\xi-z^{x, \lambda}|^{n-2s}} F^\lambda_{\ve, u} (z^{x, \lambda} ) d (z^{x, \lambda})\\
= & \gamma_{2s,n} \int_{|z-x|\geq \lambda} \frac{1}{|\xi-z|^{n-2s}} F^\lambda_{\ve, u} (z) dz\\
 & + \gamma_{2s,n} \int_{|z-x|\geq \lambda} \frac{1}{|\xi-z^{x, \lambda}|^{n-2s}} F^\lambda_{\ve, u} (z^{x, \lambda} ) \big(\frac{\lambda }{|z-x|} \big)^{2n} dz .			
\end{align*}
By using \eqref{eq-|Kelvin|} one can verify that
\begin{align*}
\frac{1}{|\xi-z^{x, \lambda}|^{n-2s}} & F^\lambda_{\ve, u} (z^{x, \lambda} ) \big(\frac{\lambda }{|z-x|} \big)^{2n} \\
&= \frac{{\lambda} ^{n-2s}}{|\xi-x|^{n-2s} |\xi^{x,\lambda} - z|^{n-2s}} 
\left[
			\begin{aligned}	
			& \varepsilon \big(\frac{2}{1+|z|^2}\big)^{2s} u(z) \\
			& + \big( \frac {2}{1+|z|^2} \big)^\sigma u^\alpha (z) 
			\end{aligned}
			\right].
\end{align*}
Hence,
		\begin{align*}
			u_{x, \lambda}(\xi)=& \gamma_{2s,n} \int_{|z-x| \geq \lambda} \frac{1}{|\xi-z|^{n-2s}} \left[
			\begin{aligned}	
			& \varepsilon\big(\frac{2 \lambda^2}{1+|z^{x, \lambda}|^2}\big)^{2s} \frac{1}{|z-x|^{4s}} u_{x, \lambda}(z)\\
			&+\big( \frac {2 \lambda^2}{1+|z^{x, \lambda}|^2} \big)^\sigma \frac{1}{|z-x|^{2 \sigma}} \; u_{x,\lambda}^\alpha (z )
			\end{aligned}
			\right] dz\\
		&+\gamma_{2s,n} \int_{|z-x| \geq \lambda} \frac{{\lambda} ^{n-2s}}{|\xi-x|^{n-2s} |\xi^{x,\lambda} - z|^{n-2s}} 
\left[
			\begin{aligned}	
			& \varepsilon \big(\frac{2}{1+|z|^2}\big)^{2s} u(z) \\
			& + \big( \frac {2}{1+|z|^2} \big)^\sigma u^\alpha (z) 
			\end{aligned}
			\right] dz.
		\end{align*}
Next, combining the above two decompositions for $u(\xi)$ and $u_{x,\lambda}(\xi)$ yields
	\begin{align*}
	u (\xi) - u_{x,\lambda} (\xi) = \gamma_{2s,n} \int_{|z-x| \geq \lambda} 
	\left( 
\begin{aligned}
& k_1 
	\left[\varepsilon \big(\frac{2}{1+|z|^2}\big)^{2s} u(z)+\big( \frac {2}{1+|z|^2} \big)^\sigma u^\alpha (z)
	\right]\\
	&-k_2 
	\left[
	\begin{aligned}
	& \varepsilon \big(\frac{2 \lambda^2}{1+|z^{x, \lambda}|^2}\big)^{2s} \frac{1}{|z-x|^{4s}}\; u_{x, \lambda}(z) \\
	&+\big( \frac {2 \lambda^2}{1+|z^{x,\lambda}|^2} \big)^\sigma \frac{1}{|z-x|^{2 \sigma}} \; u_{x, \lambda}^{\alpha}(z) 
	\end{aligned}
	\right]	
\end{aligned}
\right) dz 
 \end{align*}
		with $k_1, k_2$ being
		\[
		k_1= k_1(x, \lambda; \xi,z)= \frac 1{|\xi-z|^{n-2s}} - \big( \frac {\lambda}{|\xi-x|} \big)^{n-2s} \frac 1{|\xi^{x, \lambda}-z|^{n-2s}}
		\]
and		
		\[ 
		k_2= k_2(x, \lambda; \xi,z)= \frac{1}{|\xi-z|^{n-2s}}- \frac 1{ |\xi-z^{x,\lambda}|^{n-2s}} \big( \frac {\lambda}{|z-x|} \big)^{n-2s}.
		\]
The positivity of the kernels $k_1> 0$ and $k_2 > 0$ for any $|\xi-x|>\lambda$ and $|z-x|>\lambda$ follows from the following identities
		\begin{align*}
			\big( \frac {|\xi-x|}{\lambda} \big)^2 |\xi^{x, \lambda}-z |^2 - |\xi-z |^2 & =\frac{(|z-x|^2-\lambda^2) (|\xi-x|^2-\lambda^2)}{\lambda^2}
			\end{align*}
and
		\begin{align*}
			\big( \frac {|z-x|}{\lambda} \big)^2 |\xi-z^{x, \lambda} |^2 - |\xi-z |^2 = \frac{(|z-x|^2-\lambda^2) (|\xi-x|^2-\lambda^2)}{\lambda^2}.
		\end{align*}
In addition, the above calculation shows $k_1=k_2$. Hence, putting $k_1=k_2=K$ gives \eqref{eq-u-u=FG}. The proof is complete.
	\end{proof}
	
	In the following lemma, we show that the method of moving spheres can start from some very small $\lambda>0$. The argument is standard and depends heavily on the smoothness of $u$ around the center $x\in \R ^n$ and the asymptotic behavior of $u$ near infinity.
	
	\begin{lemma}\label{lem-MS-small}
		There exists some $\lambda_0 > 0$ such that for any $\lambda \in (0, \lambda_0)$, we have 
		\[
		u(y) \geq u_{x,\lambda} (y)
		\] 
for any $y$ satisfying $|y-x| \geq \lambda> 0$.
	\end{lemma}
	
	\begin{proof}
The proof consists of two steps. First, because $u>0$ is $C^1$, we obtain
		\[
		\nabla_y \big(|y-x|^{\frac{n-2s}2} u(y) \big)=\frac{n-2s}2 |y-x|^{\frac{n-2s}2-2} u(y) (y-x) + |y-x|^{\frac{n-2s}2}\; \nabla u(y),
		\]
which yields
		\begin{align*}
			\langle \nabla_y \big(|y-x|^{\frac{n-2s}2} u(y) \big), y-x \rangle &= |y-x|^{\frac{n-2s}2} u(y) \big( \frac{n-2s}2 + \langle \nabla \log u (y), y-x \rangle \big) \\
			&> |y-x|^{\frac{n-2s}2} u(y) \big( \frac{n-2s}2 - |\nabla \log u |_\infty |y-x| \big) .
		\end{align*}		
Hence, if $|y-x|$ is sufficiently small, say 
\[
0<|y-x| <\lambda_1: = \min \Big\{\frac{n-2s}{2 |\nabla \log u|_\infty} , 1 \Big\}, 
\]
then we must have 
\[
\langle \nabla_y \big(|y-x|^{\frac{n-2s}2} u(y) \big), y-x \rangle > 0 \quad \text{for all } 0< |y-x| < \lambda_1.
\]			
		Hence, $|y-x|^{\frac{n-2s}2} u(y)$ is increasing in the direction of $y$ with $0< |y-x| < \lambda_1$. In particular, for $0<\lambda<|y-x|<\lambda_1$, by comparing at the two points $y$ and $y^{x, \lambda}$ we must have
\[
|y-x|^{\frac{n-2s}2} u(y) \geq \big| y^{x, \lambda} -x \big|^{\frac{n-2s}2} u(y^{x, \lambda}).
\]
This shows that		
\[
u(y) \geq \Big( \frac{\big| y^{x, \lambda} -x \big|}{| y-x|} \Big)^{\frac{n-2s}2} u(y^{x, \lambda}) = u_{x,\lambda} (y)
\] 
for any $y$ satisfying $0<\lambda< |y-x| \leq \lambda_1$. In the second step, we aim to consider $|y-x| \geq \lambda_1$, and this requires the smallness of $\lambda$ leading to the threshold $\lambda_0$. To do so, first we make use of \eqref{eq-u=v}, namely
\[
		u = \big( \frac 2{1+|x|^2} \big)^{\frac {n-2s}2} (v \circ \pi_N^{-1}),
\]
and the positivity of $v$ everywhere on $\S^n$ to get
\begin{equation}\label{eq-uu>1}
u(y) \gtrsim \frac {C_1}{|y|^{n-2s}} \quad \text{for all } |y| \geq R
\end{equation}
for some constants $C_1>0$ and $R>0$. For the region $|y| \leq R$ we make use of the integral equation \eqref{eq-MAIN-Rn-integral} to get
\begin{equation}\label{eq-uu>2}
\begin{aligned}
u(y) & \geq \gamma_{2s,n} \int_{|z| \leq R} \frac 1{|y-z|^{n-2s}} \big( \frac 2{1+|z|^2} \big)^\sigma u^\alpha(z) dz\\
& \geq \frac{\gamma_{2s,n}}{(2R)^{n-2s}} \int_{|z| \leq R} \big( \frac 2{1+|z|^2} \big)^\sigma u^\alpha(z) dz > 0.	
\end{aligned}
\end{equation}
Putting the two estimates \eqref{eq-uu>1} and \eqref{eq-uu>2} together we arrive at
\[
u(y) \gtrsim \frac {C_2}{|y-x|^{n-2s}} \quad \text{for all } |y-x| \geq \lambda_1
\]
for some $C_2>0$. Hence, for small $\lambda_0 \ll \lambda_1$ and for $\lambda \leq \lambda_0$ we have
\begin{align*}
u_{x,\lambda} (y) = \big(\frac {\lambda}{|y-x| } \big)^{n-2s} u (y^{x,\lambda})
\leq \Big(\frac {\lambda_0}{|y-x| } \Big)^{n-2s} \sup_{B_{\lambda_1} (x)} u \leq u(y)
\end{align*}
for $|y-x| \geq \lambda_1 > 0$. This completes the proof.
	\end{proof}
		
For arbitrary but fixed point $x \in \R^n \setminus \{0\}$, we define
\begin{equation}\label{eq-lambda}
	\overline \lambda (x):= \sup\left\{ 
	\begin{aligned}
	0< \mu \leq \sqrt{1+|x|^2}: \quad & u_{x, \lambda}(y)\leq u(y) \quad \text{for all } \; |y-x| \geq \lambda >0\\
	& \text{and for all } \; 0< \lambda < \mu 
	\end{aligned}
	\right\}.
\end{equation}
(It is worth noting that in the definition of $\overline\lambda (x)$, we require $\mu \leq \sqrt{1+|x|^2}$, which is different from the standard method of moving spheres in \cite{Li}. This extra restriction can be easily understood by seeing Lemma \ref{lem-<1} below.) It follows from Lemma \ref{lem-MS-small} that $\overline \lambda(x)$ is well defined and $\overline \lambda(x) >0$. Next we will show that $\overline \lambda (x) \geq |x|$ for all $x\in \R ^n$. Before proving this, we need the following auxiliary lemma.
	
	\begin{lemma}\label{lem-<1}
		For any $ \lambda \in (0,\sqrt{1+|x|^2})$ there holds
\begin{equation}\label{eq-lambda/lambda<1} 
\frac{\lambda^2 \big(1+|z|^2\big)}{|z-x|^2 \big(1+|z^{x, \lambda}|^2\big) }< 1 
\end{equation}
		for any $z \in \R ^n$ satisfying $|z-x| > \lambda. $
	\end{lemma}
	
	\begin{proof}
This is elementary. Indeed, let $z \in \R ^n$ with $|z -x| >\lambda$. As $z^ {x,\lambda}= x+ \lambda^2(z-x)/|z-x|^2$, the inequality \eqref{eq-lambda/lambda<1} is equivalent to
\begin{align*}
0 < & |z-x|^2 \Big( 1+ \Big| x+ \frac{\lambda^2(z-x)}{|z-x|^2} \Big|^2 \Big) - \lambda^2(1+ |z|^2) \\
= & |z-x|^2(1+|x|^2) +2\lambda^2 \langle z-x,x\rangle + \lambda^4 - \lambda^2(1+ |z|^2) \\
= & |z-x|^2(1+|x|^2) +2\lambda^2 \langle z,x\rangle -2\lambda^2 |x|^2 + \lambda^4 - \lambda^2(1+ |z|^2) \\
= & |z-x|^2 (1+|x|^2 - \lambda^2) - \lambda^2 |x|^2 + \lambda^4 - \lambda^2 \\
=& (|z-x|^2 - \lambda^2 )(1+|x|^2 - \lambda^2) .
\end{align*}
Clearly, the right most term on the preceding estimate is strictly positive, thanks to $|z-x|> \lambda$ and $\lambda^2 < 1+|x|^2$. This completes the proof.
\end{proof} 
 
For simplicity, from now on let us denote the left hand side of \eqref{eq-lambda/lambda<1} by $L(x,\lambda,z)$, namely
\begin{equation}\label{eq-L}
L(x,\lambda,z) = \frac{\lambda^2 \big(1+|z|^2\big)}{|z-x|^2 \big(1+|z^{x, \lambda}|^2\big) }.
\end{equation}

We are now in position to estimate the threshold $\overline \lambda (x)$ for any $x \in \R^n\setminus \{0\}$.

\begin{lemma}\label{LeM=002}
	Suppose that $0< s < n/2$ and $0 \leq \sigma < 2s$. Let $u \in C( \R ^n )$ be a positive solution to the integral equation \eqref{eq-MAIN-Rn-integral}. Then, there holds
\[
\overline\lambda (x) \geq |x| \quad \text{for any }x \in \R^n \setminus \{0\}. 
\]
\end{lemma}

\begin{proof}
By way of contradiction, we suppose $\overline\lambda (x) \neq |x|$ for some $x \in \R^n \setminus \{0\}$. Then by the definition in \eqref{eq-lambda} we must have
\[
\overline\lambda (x) < |x|.
\]
For simplicity, let us write
\[
\overline\lambda =\overline\lambda (x) \quad \text{and} \quad \overline\delta =\min \big\{ 1, \frac{|x| - \overline\lambda }{2} \big\} .
\] 
Clearly, $\delta > 0$. Still by the definition of $\overline\lambda $, see \eqref{eq-lambda}, we have 
	\begin{equation}\label{5CL6-01}
		u_{x, \overline\lambda } (y) \leq u(y) \quad \text{for all $y$ satisfying } |y - x|\geq \overline\lambda . 
	\end{equation} 
However, we shall prove that 
\[
u_{x, \lambda } (y) \leq u(y) \quad \text{for any }|y - x|\geq \lambda
\] 
for some $\lambda$ slightly bigger than $\overline\lambda $. If this is true, then this violates the definition of $\overline\lambda $ in \eqref{eq-lambda}. Hence, we must have $\overline\lambda (x) \geq |x|$ as claimed. The proof consists of two steps as follows.

\noindent\textbf{Estimate of $u - u_{x, \lambda}$ outside the ball $B(x, \overline\lambda +\overline\delta)$}. To estimate $u - u_{x, \lambda}$ we first estimate $u - u_{x, \overline\lambda }$. In the region $|y-x| \geq \overline\lambda +\overline\delta$, we first claim that there exists $C_3 \in (0, 1)$ such that 
	\begin{equation}\label{Lhhdy6-03=001}
		(u - u_{x, \overline\lambda }) (y) \geq \frac{C_3}{|y -x|^{n-2 s}} \quad\text{for all } \; |y -x |\geq \overline\lambda +\overline\delta . 
	\end{equation}
To this purpose, we start from making use of the identity \eqref{eq-u-u=FG} in Lemma \ref{lem-U-U} to obtain
\begin{align*}
&(u - u_{x,\overline\lambda }) (y) \\
&= \gamma_{2s,n} \int_{|z-x| \geq \overline\lambda } K(x, \overline\lambda; y, z) \left(
	\begin{aligned}
	 &\varepsilon \big(\frac{2}{1+|z|^2}\big)^{2s} u(z)
	+\big( \frac {2}{1+|z|^2} \big)^\sigma \; u^\alpha (z)\\
	&-\varepsilon \big(\frac{2\overline\lambda ^2} {1+|z^{x,\overline\lambda }|^2}\big)^{2s} \frac{1}{|z-x|^{4s}}\; u_{x,\overline\lambda }(z) \\
	&-\big( \frac {2 \overline\lambda ^2}{1+|z^{x,\overline\lambda }|^2} \big)^\sigma \frac{1}{|z-x|^{2 \sigma}} \; u_{x, \overline\lambda }^{\alpha}(z)
	\end{aligned} \right) dz. 
\end{align*} 
We now examine the terms in the parentheses in the preceding identity. With help of \eqref{5CL6-01} and \eqref{eq-lambda/lambda<1} we observe that
\begin{equation}\label{eq-uuuu-1}
\begin{aligned}
\big(\frac{2}{1+|z|^2}\big)^{2s} u(z)
&-\big(\frac{2\overline\lambda ^2} {1+|z^{x,\overline\lambda }|^2}\big)^{2s} \frac{1}{|z-x|^{4s}} u_{x,\overline\lambda }(z) \\
& = \big(\frac{2}{1+|z|^2}\big)^{2s} \Big[ u(z) - L(x,\overline\lambda,z)^{2s} u_{x,\overline\lambda }(z) \Big] \\
& \geq \big(\frac{2}{1+|z|^2}\big)^{2s} \Big[ 1- L(x,\overline\lambda,z)^{2s} \Big] u_{x,\overline\lambda }(z) > 0
\end{aligned}
\end{equation} 
and that
\begin{equation}\label{eq-uuuu-2}
\begin{aligned}
\big(\frac{2}{1+|z|^2}\big)^\sigma u^\alpha (z)
&-\big(\frac{2\overline\lambda ^2} {1+|z^{x,\overline\lambda }|^2}\big)^\sigma \frac{1}{|z-x|^{2\sigma}} u_{x,\overline\lambda }^\alpha (z) \\
& = \big(\frac{2}{1+|z|^2}\big)^\sigma \Big[ u^\alpha (z) -L(x,\overline\lambda,z)^\sigma u_{x,\overline\lambda }^\alpha (z) \Big] \\
& \geq \big(\frac{2}{1+|z|^2}\big)^\sigma \Big[ 1-L(x,\overline\lambda,z)^\sigma \Big] u_{x,\overline\lambda }^\alpha (z) \geq 0.
\end{aligned}
\end{equation} 
Here, $L(x,\overline\lambda,z)$ is defined as in \eqref{eq-L}. Therefore,
\begin{equation}\label{eq-u-u>=}
\begin{aligned}
 (u - u_{x,\overline\lambda }) (y) 
\geq \gamma_{2s,n} \int_{|z-x| \geq \overline\lambda } K(x, \overline\lambda; y, z) L_{\ve, \overline \lambda,x}(z) dz,
\end{aligned}
\end{equation} 	
where
\begin{align*}
L_{\ve, \overline \lambda,x}(z) = & \varepsilon \big(\frac{2}{1+|z|^2}\big)^{2s} \big[ 1- L(x,\overline\lambda,z)^{2s} \big] u_{x,\overline\lambda }(z) \\
& + \big(\frac{2}{1+|z|^2}\big)^\sigma \big[ 1- L(x,\overline\lambda,z)^\sigma \big] u_{x,\overline\lambda }^\alpha (z)	.
\end{align*}
Hence, under the condition that either $\ve >0$ if $\alpha = (n+2s)/(n-2s)$ or $\ve \geq 0$ if $\alpha < (n+2s)/(n-2s)$ we deduce that
\[
L_{\ve, \overline \lambda,x}(z) > 0 \quad \text{for all } \; |z-x | \geq \overline \lambda,
\]
which tells us
\[
(u - u_{x,\overline\lambda }) (y) > 0 \quad \text{for all } \; |y-x | \geq \overline \lambda.
\]
This together with the Fatou lemma yields
\begin{align*}
\liminf_{|y| \nearrow \infty} & \big[ |y-x|^{n-2s} (u - u_{x,\overline\lambda }) (y) \big] \\
=&\gamma_{2s,n} \liminf_{|y| \nearrow \infty} \int_{|z-x| \geq \overline\lambda} |y-x|^{n-2s} K(x, \overline\lambda; y, z) L_{\ve, \overline \lambda,x}(z) dz\\
 \geq& 	 \gamma_{2s,n} \int_{|z-x| \geq \overline\lambda} \Big( 1 - \Big(\frac{\overline \lambda}{|z-x|} \Big)^{n-2s} \Big) L_{\ve, \overline \lambda,x}(z) dz \\
> & 0.	
\end{align*}
Thus, there is some small $C_4 > 0$ and some large $R \geq \overline\delta$ such that
\[
(u - u_{x, \overline\lambda }) (y) \geq \frac{C_4}{|y -x|^{n-2 s}} \quad\text{for all } \; |y -x |\geq \overline\lambda + R. 
\]
To gain \eqref{Lhhdy6-03=001}, it suffices to show that $u - u_{x, \overline\lambda }$ is bounded from below away from zero in the ring $\{ y : \overline\lambda +\overline\delta \leq |y-x| \leq \overline\lambda + R\}$. By the positivity of the kernel $K(x,\overline\lambda; y,z)$ in the region $|y-x| > \overline\lambda$ and $|z-x| > \overline\lambda$, see Lemma \ref{lem-U-U}, there is some small $C_5>0$ such that
\[
K(x, \overline\lambda; y ,z) \geq C_5
\]
for all $y$ and $z$ satisfying
\[ 
\overline\lambda +\overline\delta \leq |y-x| \leq \overline\lambda + R < 2(\overline\lambda + R) \leq |z-x| \leq 4(\overline\lambda + R). 
\]	
Using this and \eqref{eq-u-u>=} we can estimate $u - u_{x,\overline\lambda }$ from the below as follows
\begin{align*}
 (u- u_{x,\overline\lambda }) (y) 
\geq \gamma_{2s,n} C_5 \int_{2(\overline\lambda + R) \leq |z-x| \leq 4(\overline\lambda + R)} L_{\ve, \overline \lambda,x}(z) dz
= C_6
\end{align*}
for some $C_6 >0$ which possibly depends only on $x$. Putting the above estimates together we arrive at \eqref{Lhhdy6-03=001} for some $C_3$ depending on $C_4$ and $C_6$. We are now in position to estimate $u - u_{x, \lambda}$ in the region $|y-x| \geq \overline\lambda +\overline\delta$ but this requires $\lambda$ closed to $\overline \lambda$, say $\overline\lambda \leq \lambda \leq \overline\lambda +\delta_1$ for some small $\delta_1 \in (0,\overline\delta)$. Indeed, recall
\[
u_{x,\lambda} (y) = \big( \frac{\lambda}{|y-x|} \big)^{n-2s} u(y^{x,\lambda}).
\]
Hence, by continuity, there is some small $\delta_1 \in (0, \overline\delta)$ such that
\[
( u_{x, \overline\lambda } - u_{x, \lambda}) (y) \geq -\frac{C_3}{2} \frac 1{|y -x|^{n-2s}}
\]
for all $\overline\lambda \leq \lambda \leq \overline\lambda +\delta_1$ and all $|y-x| \geq \overline \lambda + \overline \delta$. Here the constant $C_3$ is as in \eqref{Lhhdy6-03=001}. This and \eqref{Lhhdy6-03=001} helps us to conclude that
\begin{equation}\label{eq-outside}
 	(u - u_{x, \lambda}) (y) =(u - u_{x, \overline\lambda}) (y)+ (u_{x,\overline\lambda} - u_{x, \lambda}) (y) \geq \frac{C_3}{2|y -x|^{n-2s}}
\end{equation}
for all $y$ satisfying $|y -x |\geq \overline\lambda + \overline\delta $ and all $\lambda$ satisfying $\overline\lambda \leq \lambda \leq \overline\lambda +\delta_1$.

\noindent\textbf{Estimate of $u - u_{x, \lambda}$ inside the ball $B(x, \overline\lambda +\overline\delta)$}. From now on, we always assume $\overline\lambda \leq \lambda \leq \overline\lambda +\delta_1$ where the constant $\delta_1$ is found in the previous step. We shall obtain \eqref{eq-inside}. Making use of \eqref{eq-u-u=FG}, \eqref{eq-uuuu-1}, and \eqref{eq-uuuu-2} we get 
\begin{align*}
(u - u_{x,\lambda}) (y) & \geq \gamma_{2s,n} \int_{|z-x| \geq \lambda} K(x, \lambda; y, z) H_{\ve, \lambda, x}(z)dz\\
&=\gamma_{2s,n} \Big( \int_{ \overline\lambda +\overline\delta \geq |z-x| \geq \lambda} + \int_{ |z-x| \geq \overline\lambda +\overline\delta } \Big) K(x, \lambda; y, z) H_{\ve, \lambda, x}(z) dz
\end{align*}
with
\begin{align*}
H_{\ve, \lambda, x}(z) =\varepsilon \big(\frac{2}{1+|z|^2}\big)^{2s} \big(u(z)-u_{x, \lambda}(z)\big) 
	+\Big( \frac {2}{1+|z|^2} \Big)^\sigma \big(u^\alpha (z)-u_{x, \lambda}^{\alpha}(z)\big) .
\end{align*}
Thanks to \eqref{eq-outside} we know that $u \geq u_{x,\lambda}$ in the region $|z-x| \geq \overline\lambda +\overline\delta $ and $\overline \lambda \leq \lambda \leq \overline \lambda + \delta_1$. This yields
\[
\int_{ |z-x| \geq \overline\lambda +\overline\delta} K(x, \lambda; y, z) H_{\ve, \lambda, x}(z) dz \geq 
\int_{\overline\lambda + 3 \geq |z-x| \geq \overline\lambda + 2} K(x, \lambda; y, z) H_{\ve, \lambda, x}(z) dz.
\]
Therefore, we can further estimate $u - u_{x,\lambda}$ as follows
\begin{align*}
(u - u_{x,\lambda}) (y) &\geq \gamma_{2s,n} \Big( \int_{ \overline\lambda +\overline\delta \geq |z-x| \geq \lambda} + \int_{\overline\lambda + 3 \geq |z-x| \geq \overline\lambda + 2} \Big) K(x, \lambda; y, z) H_{\ve, \lambda, x}(z) dz \\
& = \gamma_{2s,n} ( I + II ).
\end{align*}
Hence for some sufficiently small $\delta_2 \in (0, \delta_1)$ to be specified later we shall show that
\[
I + II \geq 0 \quad \text{for all } \lambda \leq |y-x| \leq \overline\lambda + \overline\delta, \; \overline\lambda \leq \lambda \leq \overline\lambda + \delta_2;
\]
see \eqref{eq-inside}. To see this, we estimate $I$ and $II$ term by term.

\noindent\textit{Estimate of $I$}. Using the smoothness of $u$, there exists some $C_7>0$ independent of $\delta_2 $ such that 
\begin{align}\label{eq-estimate-U-U}
\max \big\{ | u_{x, \overline\lambda }(z)- u_{x,\lambda}(z) | , | u^{\alpha}_{x, \overline\lambda }(z)- u^{\alpha}_{x,\lambda}(z) | \big\} \leq C_5 (\lambda - \overline\lambda ) \leq C_7\delta_2 
\end{align}
for all $\lambda \leq |z-x| \leq \overline\lambda + \overline\delta$ and all $\overline\lambda \leq \lambda \leq \overline\lambda +\delta_2 $. Besides, in view of \eqref{eq-K} we have the following estimate
\begin{align}\label{eq-estimate-K1}
\int_{\lambda \leq |z-x| \leq \overline\lambda +\overline\delta } & K(x,\lambda; y, z) dz 
\leq C_8 (|y-x| - \lambda) 
\end{align}
for some $C_8>0$; see Appendix \ref{apd-estimate-kernel}. As $0 \leq |z-x| -\lambda \leq \delta_2$, we obtain from \eqref{eq-estimate-K1} the following
\begin{align}\label{eq-estimate-K2}
\int_{ \lambda \leq |z-x| \leq \lambda +\delta_2 } K(x, \lambda; y, z) \big(|z-x| -\lambda\big) dz 
\leq C_8 \delta_2 (|y-x| - \lambda).
\end{align}
With help of \eqref{eq-estimate-U-U}, \eqref{eq-estimate-K1}, and \eqref{eq-estimate-K2} we are able to estimate $I$ as follows
\begin{align*}
I &= \int_{ \overline\lambda +\overline\delta \geq |z-x| \geq \lambda} K(x, \lambda; y, z) H_{\ve, \lambda, x}(z) dz 
 \geq - C_7C_8 \delta_2 (|y-x| - \lambda) .
\end{align*}

\noindent\textit{Estimate of $II$}. By \eqref{eq-outside}, there is some $C_9>0$ such that
\[
\max \big\{ u (z) - u_{x, \lambda} (z) , u^\alpha (z) - u_{x, \lambda}^\alpha (z) \big\} \geq C_9
\]
for any $z$ satisfying $\overline\lambda +2 \leq |z-x| \leq \overline\lambda +3$ and any $\overline\lambda \leq \lambda \leq \overline\lambda +\delta_2$. This leads to
\begin{align*}
II &= \int_{\overline\lambda + 3 \geq |z-x| \geq \overline\lambda + 2} K(x, \lambda; y, z) H_{\ve, \lambda, x}(z) dz \\
& \geq C_9 \int_{\overline\lambda + 3 \geq |z-x| \geq \overline\lambda + 2} K(x, \lambda; y, z) \Big(
\varepsilon \big(\frac{2}{1+|z|^2}\big)^{2s} +\big( \frac {2}{1+|z|^2} \big)^\sigma \Big) dz .
\end{align*}
Observe from \eqref{eq-K} that $K(x, \lambda; y ,z) =0$ for $y \in \partial B_\lambda(x)$ and that
\[
\langle \nabla_y K(x, \lambda; y, z), y-x \rangle \big|_{|y-x|=\lambda} = (n-2s) \frac{|z - x|^2 - |y -x|^2}{|y -z|^{n+2-2s}} >0 
\]
for all $\overline\lambda + 2 \leq |z-x| \leq \overline\lambda +3$. Therefore, there is some $C_8>0$ independent of $\delta_2$ such that
\[
K(x, \lambda; y ,z) \geq C_{10} (|y-x| - \lambda) \quad \text{for all } \overline\lambda + 2 \leq |z-x| \leq \overline\lambda +3.
\]
Putting the above estimates together we arrive at
\begin{align*}
II \geq C_7 C_8 &\int_{\overline\lambda + 3 \geq |z-x| \geq \overline\lambda + 2} \big( \frac {2}{1+|z|^2} \big)^\sigma dz .
\end{align*}

We are now in position to combine the two estimates for $I$ and $II$.
It follows from the above that for $\overline\lambda \leq \lambda \leq \overline\lambda + \delta_2$ and for $\lambda \leq |y -x| \leq \overline\lambda + \overline\delta $, 
\begin{equation}\label{eq-inside}
\begin{aligned}
\frac{(u - u_{x,\lambda}) (y) }{\gamma_{2s,n}} &=I+II\\
& \geq \Big( -C_7 C_8 \delta_2 + C_9 C_{10} \int_{\overline\lambda + 2 \leq |z -x| \leq \overline\lambda +3} \big(\frac{ 2 }{1+|z|^2}\big)^{\sigma} dz \Big) ( |y -x| -\lambda) .
\end{aligned}
\end{equation}
In view of \eqref{eq-inside} if we choose $\delta_2$ sufficiently small, then the right most hand side of \eqref{eq-inside} is positive. Together with the estimate in \eqref{eq-outside} we deduce that $(u - u_{x,\lambda}) (y) >0$. This contradicts the definition of $\overline\lambda $ in \eqref{eq-lambda}. This completes the proof of Lemma \ref{LeM=002}. 
\end{proof}

After proving Lemma \ref{LeM=002}, we have the following remarks.
\begin{remark}\label{rmk-1}
It is worth noting that the condition $x \in \R^n \setminus \{0\}$ is crucially used in the proof of Lemma \ref{LeM=002}, and it is not clear if $\overline\lambda (x)=|x|$ or not. In fact, we would like to know the exact value of $\overline \lambda (x)$ at any point $x \in \R^n \setminus \{0\}$. It turns out that with help of Theorem \ref{thm-Liouville} we are able to compute the exact value of $\overline \lambda (x)$ at any point $x \in \R^n\setminus \{0\}$; see Appendix \ref{apd-lambda}.
\end{remark}

Finally, we are in position to complete our proof of Theorem \ref{thm-Liouville}.

\begin{lemma}\label{lem-MS-stop}
The function $u$ is radially symmetric about the origin. Consequently, the function $v$ must be constant.
\end{lemma}

\begin{proof}
	Using Lemma \ref{LeM=002}, we obtain for every $x\in \R ^n \setminus \{0\}$, 
	\begin{equation}\label{SS555}
		u_{x,\lambda}(y) \leq u(y) \quad \text{for all } \; |y - x| \geq \lambda, \; 0< \lambda< |x|. 
	\end{equation}
Let $y\in \R^n \setminus \{0\}$	and $a >0$ be arbitrary but fixed. Let $\vec e \in \R^n$ be any unit vector such that
\begin{equation}\label{eq-vector}
\langle y - a \vec e , \vec e\rangle \leq 0.
\end{equation}
For any number $R >a$ if we set $\lambda = R - a$ and $x=R\vec e$, then $0<\lambda< |x|$ and with help of \eqref{eq-vector} we get
\[
|y-x|^2 = |y - a \vec e -\lambda \vec e|^2 =\lambda^2 + |y- a \vec e |^2 -2\lambda\langle y - a \vec e , \vec e\rangle \geq \lambda^2.
\]
Therefore, we can apply \eqref{SS555} to get
\begin{align*}
	u(y)\geq u_{x,\lambda}(y) & =\big( \frac{\lambda}{|y - x|}\big)^{n-2s} u\big( x + \frac{\lambda^2(y -x)}{|y -x|^2} \big) \\
	& = \big(\frac{R-a}{|y-R\vec e|}\big)^{n-2s} u\big(R\vec e+\frac{(R-a)^2(y-R\vec e)}{|y-R\vec e|^2}\big).
\end{align*}
Notice that	
\[
	\aligned
	R\vec e+\frac{(R-a)^2(y-R\vec e)}{|y-R\vec e|^2} &= \frac{R|y-R\vec e|^2e+(R-a)^2(y-R\vec e)}{|y-R\vec e|^2}\\
	&=
	 \frac{R(|y|^2-2R\langle y,\vec e\rangle+R^2)\vec e+ (R^2-2Ra+a^2)(y-R\vec e)}{|y-R\vec e|^2}
	 \endaligned
\]
Hence, by letting $R$ to infinity, we get
\[
R\vec e+\frac{(R-a)^2(y-R\vec e)}{|y-R\vec e|^2} \to y-2(\langle y, \vec e \rangle-a)\vec e .
\]
Obviously,
	 $$
	 \aligned
	 \big(\frac{R-a}{|y-R\vec e|}\big)^{n-2s} \to 1
	 \endaligned
	 $$
as $R$ goes to infinity. Hence, by the continuity of $u$, we arrive at
	 \begin{equation}\label{SS666}
	 	u(y) \geq u(y - 2( \langle y, \vec e \rangle - a )\vec e ).
	 \end{equation}
Since the inequality \eqref{SS666} above holds for arbitrary $a>0$, we let $a \searrow 0$ to get 
\[
u(y)\geq u(-y) \quad \text{for all } \; y \in \R^n \setminus \{0\}.
\] 
Since $y \in \R^n\setminus \{0\}$ is arbitrary, it follows from the preceding inequality that $u$ is radially symmetric about the origin. Having the symmetry of $u$ one can quickly conclude that $v$ must be constant. Indeed, using the relation
\[
		u = \big( \frac 2{1+|x|^2} \big)^{\frac {n-2s}2} (v \circ \pi_N^{-1}) 
		\quad \text{in } \R^n
\]
we know that the function $v$ depends only the last coordinate $x_{n+1}$ in $\R^{n+1}$. However, as the $x_{n+1}$-axis is arbitrarily chosen, the function $v$ must be constant. This completes the proof.
\end{proof}

 
\section{Application to the sharp critical fractional Sobolev inequality}
\label{sec-Sobolev}

This section is devoted to a proof of Theorem \ref{thm-Sobolev} which concerns a sharp subcritical/critical Sobolev inequality \eqref{eq-sSobolev} for non-negative functions, namely
\[
\int_{\S^n} v \gjms (v) d\mu_{g_{\S^n}}
\geq \frac{\Gamma (n/2 + s)}{\Gamma (n/2 - s )} | \S^n|^\frac{\alpha-1}{\alpha+1}
\Big( \int_{\S^n} v^{\alpha+1} d\mu_{g_{\S^n}} \Big)^\frac{2}{\alpha+1}
\] 
with $0<\alpha \leq (n+2s)/(n-2s)$. We note by the H\"older inequality that
\[
\int_{\S^n}v^{\alpha+1} d\mu_{g_{\S^n}}
\leq |\S^n|^{1-\frac{(\alpha+1)(n-2s)}{2n}} 
\Big(  \int_{\S^n}v^\frac{2n}{n-2s} d\mu_{g_{\S^n}}  \Big)^\frac{(\alpha+1)(n-2s)}{2n},
\]
which implies
\[
| \S^n|^\frac{\alpha-1}{\alpha+1}
\Big( \int_{\S^n} v^{\alpha+1} d\mu_{g_{\S^n}} \Big)^\frac{2}{\alpha+1}
\leq  | \S^n|^\frac{2s}{n} 
\Big( \int_{\S^n} v^\frac{2n}{n-2s} d\mu_{g_{\S^n}} \Big)^\frac{n-2s}{n}.
\] 
Hence, the subcritical case of \eqref{eq-sSobolev} can be derived directly from the critical case of \eqref{eq-sSobolev}. Of course, it is clear that equality in \eqref{eq-sSobolev} occurs if $v$ is any constant function. Therefore, in the rest of this section, it suffices to investigate the critical case of \eqref{eq-sSobolev}, namely we shall prove the following sharp inequality
\[
\int_{\S^n} v \gjms (v) d\mu_{g_{\S^n}}
\geq \frac{\Gamma (n/2 + s)}{\Gamma (n/2 - s )} | \S^n|^\frac{2s}{n} 
\Big( \int_{\S^n} v^\frac{2n}{n-2s} d\mu_{g_{\S^n}} \Big)^\frac{n-2s}{n}.
\] 
Set
\[
\mathcal S = \Q |\S^n|^{2s/n} .
\]
Since we shall make use of a limit process, it is freely to consider $\ve \in (0, 1)$. Now we consider the  variational problem \eqref{eq-SS-O}, i.e.
\[
\mathcal S_\ve = \inf_{v \in \mathcal W }\int_{\S^n} \big[ v \gjms (v) - \ve \Q v^2 \big] d\mu_{g_{\S^n}}
\]
within the set
\[
\mathcal W =\Big\{0\leq v \in H^s (\S^n) : \int_{\S^n} v^\frac{2n}{n-2s} d\mu_{g_{\S^n}} = 1 \Big\}.
\]
The set $\mathcal W$ is not empty because $|\S^n|^{-\frac{n-2s}{2n}} \in \mathcal W$. Besides, we also have
\[
|\S^n|^{-\frac{n-2s}{2n}} \gjms (|\S^n|^{-\frac{n-2s}{2n}}) - \ve \Q |\S^n|^{-\frac{n-2s}{n}} =(1- \ve )\Q |\S^n|^{-\frac{n-2s}{n}} \ne 0,
\]
which helps us to conclude from \eqref{eq-SS-O} that
\[
\mathcal S_\ve \leq (1- \ve )\Q |\S^n|^{2s/n}<+\infty,
\] 
however, $\mathcal S_\ve $ could be $-\infty$. Eventually we are able to show that $\mathcal S_\ve = (1- \ve )\Q |\S^n|^{2s/n}$, but at the moment, we show that $\mathcal S_\ve$ is finite and is achieved by some smooth non-negative function $v_\ve$. 

\begin{lemma}\label{lem-SSexistence}
The constant $\mathcal S_\ve$ in \eqref{eq-SS-O} is finite and there exists some non-negative $v_\ve \in C^\infty (\S^n)$ such that 
\[
\int_{\S^n} \big[ v_\ve \gjms (v_\ve) - \ve \Q v_\ve^2 \big] d\mu_{g_{\S^n}} = \mathcal S_\ve \Big( \int_{\S^n} v_\ve^\frac{2n}{n-s} d\mu_{g_{\S^n}} \Big)^\frac{n-2s}{n}.
\]
In particular, $v_\ve$ solves \eqref{EqVe}, i.e. 
\[
\gjms (v_\ve) = \ve \Q v_\ve +\mathcal S_\ve v_\ve^\frac{n+2s}{n-2s}  \quad \text{on $\S^n$}.
\]
\end{lemma}

\begin{proof}
 Let us write 
\[
A_{2s, \ve}(v) := \int_{\S^n}\big[ v \gjms (v) - \ve \Q v^2 \big] d\mu_{g_{\S^n}}.
\]
Note that, $\alpha_{2s,n}(l)\geq 0$ for all $l$ as $2s<n$ and $\alpha_{2s,n}(l)$ grows like $l^{2s}$ for large $l$ by Stirling's formula; see \eqref{eq-Gamma/Gamma} in Appendix \ref{apd-alpha}. Due to the formula \eqref{ProjectGJMS} and the fact that the remaining finite rank terms are bounded in $L^2(\S^n)$, we obtain for all non-negative $ v\in H^s(\S^n)$ that
\begin{equation}\label{LowerBoundGJMS}
\int_{\S^n} v\gjms(v) d\mu_{g_{\S^n}} \geq C_{11}\|v\|_{H^s(\S^n)}^2 -C_{12} \|v\|_{L^2(\S^n)}^2
\end{equation}
with two positive constants $C_{11}$ and $C_{12}$. 
This and the H\"older inequality imply
\begin{align*}
A_{2s,\ve}(v)& \geq
- \big(C_{12}+ \ve \Q \big) \int_{\S^n}v^2 d\mu_{g_{\S^n}}\\
&\geq 
- \big(C_{12}+ \ve \Q \big) |\S^n|^\frac{2s}{n}\Big( \int_{\S^n} v^\frac{2n}{n-2s} d\mu_{g_{\S^n}} \Big)^\frac{n-2s}{n}.
\end{align*}
 From this we get $\mathcal S_\ve>-\infty$. In fact, thanks to $\ve>0$ one should have
\[
\mathcal S_\ve < \mathcal S.
\]
This is trivial because by testing with the constant $|\S^n|^{-\frac{n-2s}{2n}}$ function on $\S^n$ one should have
\begin{align*}
\mathcal S_\ve & \leq \int_{\S^n} (1- \ve) \Q |\S^n|^{-\frac{n-2s}{n}} d\mu_{g_{\S^n}} 
 <\Q |\S^n|^\frac{2s}{n} =\mathcal S . 
\end{align*}

Having this if we let $(v_k)_{k \geq 1}$ be a minimizing sequence for $\mathcal S_\ve$, then using \eqref{LowerBoundGJMS} again, we have $(v_k)_{k\geq 1}$ is bounded in $H^s(\S^n)$. By the lower semicontinuity, we have
\[
A_{2s,\ve}(\lim_{k\to \infty} v_k) \leq \lim_{k\to \infty} A_{2s,\ve}(v_k)=1. 
\] 
Therefore, using the concentration-compactness principle, see e.g. \cite[Theorem 4]{SM} for integer case and \cite[Theorem 1.1]{BSS18} for fractional case, there are two non-negative Borel regular measures $\mu$ and $\nu$ on $\S^n$ and a non-negative function $v_\ve \in H^{s}(\S^n)$ such that
\begin{equation}\label{eq-weakV_k}
v_k \rightharpoonup v_\ve \quad \text{weakly in $H^{s}(\S^n)$ and a.e. in $\S^n$},
\end{equation}
that
\begin{equation}\label{eq-->mu}
|(-\Delta_{\S^n})^{s/2} v_k|^2 d\mu_{g_{\S^n}} \rightharpoonup \mu \quad \text{weakly in the sense of measures},
\end{equation}
and that
\begin{equation}\label{eq-->nu}
 v_k^\frac{2n}{n-2s} \rightharpoonup \nu \quad \text{weakly in the sense of measures}, 
\end{equation}
up to a subsequence.  In addition, there is an at most countable set $\mathcal I$, a family of distinct points $\{ x_i : i\in \mathcal I\}\ \subset \S^n$, families of non-negative weights $\{\alpha_i : i \in \mathcal I \}$ and $\{\beta_i : i \in \mathcal I \}$ such that
\begin{equation}\label{eq--=nu}
\nu = v_\ve^\frac{2n}{n-2s}d\mu_{g_{\S^n}} + \sum_{i \in \mathcal I} \alpha_i \delta_{x_i},
\end{equation}
and
\begin{equation}\label{eq->=nu}
\mu \geq |(-\Delta_{\S^n})^{s/2} v_\ve|^2 d\mu_{g_{\S^n}} + \sum_{i \in \mathcal I} \beta_i \delta_{x_i},
\end{equation}
and
\begin{equation}\label{AlphaBeta_i}
\alpha_i^\frac{n-2s}{n} \leq \frac{\beta_i}{ \mathcal S} \quad \text{for all } i \in \mathcal I.
\end{equation}
As $v_k \in\mathcal W$ and $v_k^\frac{2n}{n-2s} \rightharpoonup \nu$ weakly in the sense of measures, testing \eqref{eq--=nu} with suitable constant functions gives
\[
\int_{\S^n} d\nu = 1,
\]
which, by making use of \eqref{eq--=nu}, further implies
\begin{equation}\label{eq-S=S}
1 = \int_{\S^n} v_\ve^\frac{2n}{n-2s} d\mu_{g_{\S^n}} + \sum_{i \in \mathcal I} \alpha_i .
\end{equation}
In particular, there holds $\alpha_i < 1$ for all $i \in \mathcal I$. In view of the inequality \eqref{eq->=nu} for $\mu$ gives
\begin{equation}\label{EstSveSvk}
\begin{split}
\mathcal S_\ve & = \lim_{k \to +\infty} \int_{\S^n} \big[ v_k \gjms (v_k) - \ve \Q v_k^2 \big] d\mu_{g_{\S^n}}\\
 & = \lim_{k \to +\infty} \int_{\S^n} \big[ v_k \gjms (v_k) -|(-\Delta_{\S^n})^{s/2} v_k|^2\big] d\mu_{g_{\S^n}} \\
 & \qquad + \lim_{k \to +\infty}\int_{\S^n}\big[ |(-\Delta_{\S^n})^{s/2}v_k|^2- \ve \Q v_k^2 \big] d\mu_{g_{\S^n}} \\
& \geq \lim_{k \to +\infty} \int_{\S^n} \big[ v_k \gjms (v_k) -|(-\Delta_{\S^n})^{s/2} v_k|^2\big] d\mu_{g_{\S^n}} \\
 & \qquad + \int_{\S^n} \big[|(-\Delta_{\S^n})^{s/2} v_\ve|^2 - \ve \Q v_\ve^2 \big] d\mu_{g_{\S^n}}+ \sum_{i \in \mathcal I} \beta_i.
\end{split}
\end{equation}
Set $w_k = v_k-v_\ve$, we now show that 
\begin{equation}\label{eqLimDsW_k=0}
 \lim_{k \to +\infty} \int_{\S^n} \big[ w_k \gjms (w_k) -|(-\Delta_{\S^n})^{s/2} w_k|^2\big] d\mu_{g_{\S^n}} =0.
\end{equation}
Let $l\in \N_0$ and $Y_l$ be the spherical harmonic of degree $l$ on $\S^n$. Since $(Y_l)_{l\in \N_0}$ is a basis of $L^2(\S^n)$, we can write 
\[
w_k = \sum_{l\in \N_0} b_{kl} Y_l.
\] 
Owing to the weak convergence of $w_k$ to 0 in $H^s(\S^n)$, see \eqref{eq-weakV_k}, we have, for each $l \in \N_0$, that $b_{kl} \to 0$ as $k\to \infty$. Note that $(w_k)_{k\geq1}$ is uniformly bounded in $H^s(\S^n)$, yielding $\sum_{l\in \N_0} b_{k,l}^2 l^{2s}$ is uniformly bounded in $k$ by $C_{13}$. Note also that
\[
(-\Delta_{\S^n})^{s/2} Y_l = l^{s/2}(l+n-1)^{s/2} Y_l
\]
and that
\[\gjms(Y_l)= \alpha_{2s,n}(l) Y_l.\]
Hence, 
\[
\Big| \int_{\S^n} \big[ w_k \gjms (w_k) -|(-\Delta_{\S^n})^{s/2} w_k|^2\big] d\mu_{g_{\S^n}}\Big| \leq \sum_{l\in \N_0} \big|\alpha_{2s,n}(l)-l^s(l+n-1)^s \big| b_{k,l}^2
\]
Thanks to  \eqref{eq-Gamma/Gamma} again, we have that for large $l$,
\[
\begin{split} 
\alpha_{2s,n}(l)-l^s(l+n-1)^s &=  l^{2s} + s (n-1) l^{2s-1}  +\frac{s}6 \big(3(2s-1)(n-1)^2-4s^2+3 \big) l^{2s-2}\\
&\qquad- \big( l^{2s}+s (n-1) l^{2s-1} +  \frac12 s(s-1)(n-1)^2 l^{2s-2}\big)+ O(l^{2s-3})_{l \nearrow +\infty}\\
&= s\big( \frac12(s-2) (n-1)^2-\frac23 s^2+\frac13\big) l^{2s-2}+O(l^{2s-3})_{l \nearrow +\infty}.
\end{split}
\]
Consequently, for $h$ sufficiently small, we take 
\[
m_0^2>\frac{2s}h \big| \frac12 (s-2)(n-1)^2-\frac23 s^2+\frac13\big|C_{13} 
\] to obtain
\[
\begin{split}
\int_{\S^n} & \big[ w_k \gjms (w_k) -|(-\Delta_{\S^n})^{s/2} w_k|^2\big] d\mu_{g_{\S^n}}\\
 &\leq \sum_{l=0}^{m_0} \big|\alpha_{2s,n}(l)-l^s(l+n-1)^s \big| b_{k,l}^2+2s\big| \frac12 (s-2)(n-1)^2-\frac23 s^2+\frac13\big| \sum_{l>m_0} b_{k,l}^2 l^{2s-2}\\
&\leq \sum_{l=0}^{m_0} \big|\alpha_{2s,n}(l)-l^s(l+n-1)^s\big| b_{k,l}^2+ \frac{2s\big| \frac12 (s-2)(n-1)^2-\frac23 s^2+\frac13\big| C_{13}}{m_0^2} \\
 &\leq \sum_{l=0}^{m_0} \big|\alpha_{2s,n}(l)-l^s(l+n-1)^s \big| b_{k,l}^2+h.
\end{split}
\]
Sending $k\to \infty$, we deduce that 
\[
\Big| \lim_{k \to +\infty} \int_{\S^n} \big[ w_k \gjms (w_k) -|(-\Delta_{\S^n})^{s/2} w_k|^2\big] d\mu_{g_{\S^n}}\Big| \leq h
\]
for any $h$ sufficiently small. The limit \eqref{eqLimDsW_k=0} thus follows. Using the weak convergence $w_k = v_k-v_\ve \rightharpoonup 0$ in $H^s (\S^n)$ we deduce that
\begin{align*}
&\int_{\S^n} \big[ v_k \gjms (v_k) -|(-\Delta_{\S^n})^{s/2} v_k|^2\big] d\mu_{g_{\S^n}} \\
&= \int_{\S^n} \big[ v_\ve \gjms (v_\ve) -|(-\Delta_{\S^n})^{s/2} v_\ve|^2\big] d\mu_{g_{\S^n}} +\int_{\S^n} \big[ w_k \gjms (w_k) -|(-\Delta_{\S^n})^{s/2} w_k|^2\big] d\mu_{g_{\S^n}}\\
&\qquad+o(1)_{k \nearrow +\infty}.
\end{align*}
Hence, combining this with \eqref{EstSveSvk} and \eqref{eqLimDsW_k=0} gives
\[
\begin{split}
\mathcal S_\ve & \geq \lim_{k \to +\infty} \int_{\S^n} \big[ v_k \gjms (v_k) -|(-\Delta_{\S^n})^{s/2} v_k|^2\big] d\mu_{g_{\S^n}} \\
 & \qquad + \int_{\S^n} \big[|(-\Delta_{\S^n})^{s/2} v_\ve|^2 - \ve \Q v_\ve^2 \big] d\mu_{g_{\S^n}}+ \sum_{i \in \mathcal I} \beta_i\\
& = \int_{\S^n} \big[ v_\ve \gjms (v_\ve) - \ve \Q v_\ve^2 \big] d\mu_{g_{\S^n}}+ \sum_{i \in \mathcal I} \beta_i.
\end{split}
\]
Using \eqref{AlphaBeta_i} and the boundedness $\alpha_i<1$ for any $i\in \mathcal I$, we get further
\[
\begin{split}
\mathcal S_\ve &\geq \mathcal S_\ve \int_{\S^n} v_\ve^{\frac{2n}{n-2s}} d\mu_{g_{\S^n}}+ \mathcal S \sum_{i \in \mathcal I} \alpha_i^{\frac{n-2s}n} \geq \mathcal S_\ve \int_{\S^n} v_\ve^{\frac{2n}{n-2s}} d\mu_{g_{\S^n}}+ \mathcal S \sum_{i \in \mathcal I} \alpha_i.
\end{split}
\]
Thanks to \eqref{eq-S=S} we eventually arrive at
\[
0 \geq \big( \mathcal S - \mathcal S_\ve \big) \sum_{i \in \mathcal I} \alpha_i \geq 0 .
\]
Keep in mind that $\mathcal S_\ve < \mathcal S$. Therefore, the only possibility for which the above inequalities occur is $\alpha_i = 0$ for all $i \in \mathcal I$. This implies that no concentration occurs, that is equivalent to saying that
\[
v_k \to v_\ve \quad \text{strongly in $H^s (\S^n)$}.
\]
In particular, $v_\ve \in \mathcal W$, and standard arguments show that $v_\ve$ solves
\[
\gjms (v_\ve) = \ve \Q v_\ve + \mathcal S_\ve v_\ve^\frac{n+2ss}{n-2s} \quad \text{in $\S^n$}.
\]
Hence
\begin{align*}
\mathcal S_\ve & \leq \int_{\S^n} \big[ v_\ve \gjms (v_\ve) - \ve \Q v_\ve^2 \big] d\mu_{g_{\S^n}}
 = \mathcal S_\ve \int_{\S^n} v_\ve^\frac{n+2s}{n-2s} d\mu_{g_{\S^n}} 
= \mathcal S_\ve .
\end{align*}
Hence, the constant $\mathcal S_\ve$ is attained by $v_\ve\geq 0$. This completes the proof.
\end{proof}

\begin{remark}
When $s=m$ is an integer, part of the above proof can be simplified. In this case, $\gjms$ is a local operator and is a linear combination of finite terms  (see \eqref{GJMS-integer}). Precisely, we do not have to go through \eqref{eqLimDsW_k=0} because
\begin{align*}
\lim_{k\to \infty} \int_{\S^n} v_k \mathbf P_n^{2m} (v_k) d\mu_{g_{\S^n}} & = \lim_{k\to \infty} \int_{\S^n} \big[ |\nabla^m v_k|^2 + \sum_{j=0}^{m-1} a_j |\nabla^i v_k|^2 \big] d\mu_{g_{\S^n}} \\
& \geq \int_{\S^n} |\nabla^m v_\ve|^2 d\mu_{g_{\S^n}} + \sum_{i \in \mathcal I} \beta_i + \sum_{j=0}^{m-1} a_j \int_{\S^n} |\nabla^i v_\ve|^2 d\mu_{g_{\S^n}} \\
& = \int_{\S^n} v_\ve \mathbf P_n^{2m} (v_\ve) d\mu_{g_{\S^n}} + \sum_{i \in \mathcal I} \beta_i ,
\end{align*}
thanks to \eqref{eq->=nu} and the strong convergence $\nabla^j v_k \to \nabla^j v_\ve$ in $L^2 (\S^n)$ for all $0 \leq j \leq m-1$. From this we immediately get
\[
\mathcal S_\ve \geq \int_{\S^n} \big[ v_\ve \mathbf P_n^{2m} (v_\ve) - \ve Q_n^{2m} v_\ve^2 \big] d\mu_{g_{\S^n}}+ \sum_{i \in \mathcal I} \beta_i.
\]
\end{remark}
Having Lemma \ref{lem-SSexistence}  in hand, we are able to prove Theorem \ref{thm-Sobolev} as we shall do now. By seeing our Liouville type result in Theorem \ref{thm-Liouville}, this is the place we need the smallness of $\ve$. 

\begin{proof}[Proof of Theorem \ref{thm-Sobolev}]
Let $\ve \in (0,1)$. By Lemma \ref{lem-SSexistence}, there is some non-negative, smooth function $v_\ve$ satisfying
\[
\gjms (v_\ve) = \ve \Q v_\ve +\mathcal S_\ve v_\ve^\frac{n+2s}{n-2s} \quad \text{in $\S^n$}.
\]
Then, it follows from Theorem \ref{thm-Liouville} that $v_\ve$ must be constant. Hence, on one hand, as $Q_n^{2s} = \gjms (1)$, we can compute to get
\[
\mathcal S_\ve = ( 1- \ve) \Q |\S^n|^{2s/n},
\]
on the other hand, by the definition of $\mathcal S_\ve$ we get
\begin{align*}
\frac{ \int_{\S^n} \big[ v \gjms (v) - \ve \Q v^2 \big] d\mu_{g_{\S^n}}}{\Big( \int_{\S^n} v^\frac{2n}{n-2s} d\mu_{g_{\S^n}} \Big)^{\frac{n-2s}{n}}}
 \geq ( 1- \ve) \Q | \S^n|^{2s/n}
\end{align*}
for any non-negative $v \in H^s (\S^n)$. Now letting $\ve \searrow 0$ we obtain
\begin{equation}\label{IneSoboblev}
\int_{\S^n} v \gjms (v) d\mu_{g_{\S^n}}
\geq \Q | \S^n|^\frac{2s}{n}
\Big( \int_{\S^n} v^\frac{2n}{n-2s} d\mu_{g_{\S^n}} \Big)^\frac{n-2s}{n} 
\end{equation}
for any non-negative $v \in H^s (\S^n)$.  This completes the proof of the critical case of \eqref{eq-sSobolev}, hence completing the proof of Theorem \ref{thm-Sobolev}.
\end{proof}

	
	\section*{Acknowledgments}

This work was initiated when QAN was visiting the Vietnam Institute for Advanced Study in Mathematics (VIASM) in 2022. QAN would like to thank VIASM for hospitality and financial support. QAN would also like to thank Ali Hyder for useful discussion on this topic. TTN is especially indebted to Professor Jean-Marc Delort and Professor Olivier Lafitte for their enthusiastic encouragement. We thank Van Hoang Nguyen for pointing out an inaccuracy in Section 4 in an earlier version of this paper. The research of QAN is funded by Vietnam National Foundation for Science and Technology Development (NAFOSTED) under grant number 101.02-2021.24 ``On several real and complex differential operators''.  Part of this work was announced in the report \cite{QuynhLe}.


\appendix 
\section{Estimate (\ref{eq-estimate-K1}) for the kernel $K$}
\label{apd-estimate-kernel}

This appendix is devoted to the proof of \eqref{eq-estimate-K1}, namely we prove that
\begin{align*}
\int_{\lambda \leq |z-x| \leq \overline\lambda +\overline\delta } K(x,\lambda; y, z) dz \leq C \big( |y-x| - \lambda \big) 
\end{align*}
for some $C>0$. More or less this estimate is standard, but we re-mention it for completeness. Our argument depends on the following two ingredients. First we mention the elementary inequality
\[
\Big| \frac 1{x^p} -\frac 1{y^p} \Big| \leq p |x-y| \max\Big\{ \frac 1{x^{p+1}}, \frac 1{y^{p+1}} \Big\}
\]
for any $x,y>0$ and any $p>0$. Next we observe that with $\kappa < n$, if $y \in \overline B(x,R)$, then the following integral
\begin{equation}\label{eq-eq}
\int_{\lambda \leq |z-x| \leq \overline\lambda +\overline\delta } \frac{1}{|y -z|^\kappa} dz \leq C(\overline\lambda, \overline\delta,|x|,R)
\end{equation}
is bounded from above, whose bound $C$ depends on $\overline\lambda$, $\overline\delta$, $|x|$, and $R$. This is simple because 
\begin{align*}
\int_{\lambda \leq |z-x| \leq \overline\lambda +\overline\delta } \frac{1}{|y -z|^\kappa} dz 
& =\int_{\{y+z : \lambda \leq |z-x| \leq \overline\lambda +\overline\delta \}} \frac{dz}{|z|^\kappa} 
 \leq \int_{B( \overline\lambda +\overline\delta +|x| +R)} \frac{ dz}{|z|^\kappa}
 =: C
\end{align*}
by triangle inequality. By using \eqref{eq-K} we obtain
\begin{align*}
\int_{\lambda+\delta_2 \leq |z-x| \leq \overline\lambda +\overline\delta } & K(x,\lambda; y, z) dz \\
\leq & \int_{\lambda \leq |z-x| \leq \overline\lambda +\overline\delta } K(x,\lambda; y, z) dz\\
 \leq & \int_{\lambda \leq |z -x| \leq \overline\lambda +\overline\delta } \Big| \frac{1}{|y -z|^{n-2s}} - \frac{1}{|y^{x,\lambda} -z|^{n-2s}}\Big| dz \\ 
&+\int_{\lambda \leq |z-x| \leq \overline\lambda +\overline\delta } \Big| \Big(\frac{\lambda}{|y-x|}\Big)^{n-2s} -1 \Big| \frac{1}{|y^{x,\lambda} -z|^{n-2s}} dz\\
=& I_1 + I_2.
\end{align*}
Using the above elementary inequality we immediately get
\begin{align*}
\Big| \frac{1}{|y -z|^{n-2s}} &- \frac{1}{|y^{x,\lambda} -z|^{n-2s}}\Big|\\
&\leq (n-2s) |y-y^{x,\lambda}| \max\Big\{\frac{1}{|y -z|^{n-2s+1}}, \frac{1}{|y^{x,\lambda} -z|^{n-2s+1}}\Big\},
\end{align*}
which, after making use of \eqref{eq-eq} twice, allows us to write
\begin{align*}
I_1 \leq C |y-y^{x,\lambda}|
\end{align*}
for some $C>0$. (To be able to apply \eqref{eq-eq} we note that $|y^{x,\lambda} - x| = \lambda^2/|y-x| \leq \lambda \leq \overline\lambda +\overline \delta$ for $y$ being satisfied $\lambda \leq |y-x| \leq \overline\lambda +\overline \delta$.) For the term $I_2$, we note that
\begin{align*}
\Big| \Big(\frac{\lambda}{|y-x|}\Big)^{n-2s} -1 \Big|
& =\lambda^{n-2s} \Big| \frac 1{|y-x|^{n-2s}} - \frac 1{\lambda^{n-2s}} \Big|\\
& \leq (n-2s) \lambda^{n-2s} \big| |y-x| - \lambda \big| \max \Big\{ \frac 1{|y-x|^{n-2s}} , \frac 1{\lambda^{n-2s}} \Big\}\\
& \leq (n-2s) \big( |y-x| - \lambda \big)
\end{align*}
as $|y-x| \geq \lambda$. Thus, applying \eqref{eq-eq} once gives
\begin{align*}
I_2 \leq C \big( |y-x| - \lambda \big) 
\end{align*}
for some $C>0$. Putting the above estimates for $I_1$ and $I_2$ together we arrive at
\begin{align*}
\int_{\lambda \leq |z-x| \leq \overline\lambda +\overline\delta } K(x,\lambda; y, z) dz & \leq C |y-y^{x,\lambda}| + C \big( |y-x| - \lambda \big) \\
&=C \Big| 1- \frac{\lambda^2}{|y-x|^2} \Big| \big|y-x\big| + C \big( |y-x| - \lambda \big) \\
& \leq C \big( |y-x| - \lambda \big) 
\end{align*}
for some $C>0$. 
	
\section{Asymptotic behavior of $\alpha_{2s,n}(l)-l^{2s}-s(n-1)l^{2s-1}$ for large $l$}
\label{apd-alpha}

Recall from \eqref{Alpha} the following
\[
\alpha_{2s,n}(l)= \frac{\Gamma(l+n/2+s)}{\Gamma(l+n/2-s)}.
\]
In \cite[page 13]{FKT}, the authors claim that $\alpha_{2s,n}(l)$ grows like $l^{2s}$ for large $l$ in the sense that
\[
\alpha_{2s,n}(l)=l^{2s} + o(1)_{l \nearrow +\infty}.
\]
To serve our purpose, see the proof of Lemma \ref{lem-SSexistence}, we need a refiner asymptotic behavior for $\alpha_{2s,n}(l)$ for large $l$, and this is the content of this appendix. To be more precise, we prove that
\begin{equation}\label{eq-Gamma/Gamma}
\alpha_{2s,n}(l)= l^{2s} + (n-1) s l^{2s-1}  +\frac{s}6 \big(3(2s-1)(n-1)^2-4s^2+3 \big) l^{2s-2}+ O(l^{2s-3})_{l \nearrow +\infty}.
\end{equation}
This can be derived by making use of  the following formula for any positive numbers $a$ and $b$, 
\begin{equation}\label{FractionGamma}
\begin{split}
\frac{\Gamma(x+b)}{\Gamma(x+a)} &= x^{b-a} + \frac12 (b-a)(b+a-1) x^{b-a-1}+ \frac{b-a}{12} 
\left(\begin{split}
&3(b-a)(b+a-1)^2\\
&-4(b^2+ba+a^2)+6(b+a)
\end{split}\right) x^{b-a-2}\\
&\qquad \qquad\qquad+ O(x^{b-a-3})_{x\nearrow +\infty},
\end{split}
\end{equation}
which refines the formula mentioned in \cite{ET51}
\[
\frac{\Gamma(x+b)}{\Gamma(x+a)} = x^{b-a} + \frac 12 (b-a)(b+a-1) x^{b-a-1} + O(x^{b-a-2})_{x\nearrow +\infty}.
\]
Indeed, using \eqref{FractionGamma} with $b=\frac{n}2+ s$ and $a=\frac{n}2-s>0$, we deduce \eqref{eq-Gamma/Gamma}.

Since there is no direct proof of \eqref{FractionGamma} in \cite{ET51} and to make our paper self-contained, we include here a proof for the reader's convenience. Our starting point is the following Stirling's formula 
\[
\ln\Gamma(x)= x\ln x-x +\frac12 \ln\big( \frac{2\pi}x\big) +\sum_{n=1}^{N}\frac{B(n)}{2n(2n-1)x^{2n-1}}+ O(x^{-2N})_{x\nearrow +\infty},
\]
where $B(n)$ is the $n$-th Bernoulli number. 
Hence, for any $b>0$, we have
\[
\begin{split}
\ln \Gamma(x+b) &=\frac12\ln(2\pi)+ \big(x+b -\frac12\big) \ln(x+b)-(x+b)+ \frac{B(2)}{2(x+ b)}\\
&\qquad+\frac{B(3)}{12(x+b)^3}+ O(x^{-4})_{x \nearrow +\infty}.
\end{split}
\]
Note that for $b>0$,
\[
\ln(x+b) =\ln x +\frac{b}x - \frac{b^2}{2x^2}+ \frac{b^3}{3x^3} +O(x^{-4})_{x\nearrow +\infty}.
\]
That implies 
\[
\begin{split}
\ln \Gamma(x+b) &= \frac12\ln(2\pi) + \big(x+b-\frac12\big) \Big( \ln x +\frac{b}x - \frac{b^2}{2x^2} + \frac{b^3}{3x^3} +O(x^{-4})_{x\nearrow +\infty}\Big)\\
&\qquad -(x+b) + \frac{B(2)}{2x} +\frac{B(3)}{12(x+b)^3}+ O(x^{-4})_{x \nearrow +\infty} \\
&= \frac12\ln(2\pi) -x + \big(x+b-\frac12\big)\ln x+ \frac{b(b-1)+B(2)}{2x}- \frac{b^2(2b-3)}{12x^2}+ O(x^{-3})_{x \nearrow +\infty}.
\end{split}
\]
For any positive numbers $a$ and $b$, we deduce that 
\[\begin{split}
\ln \Gamma(x+b) -\ln \Gamma(x+a ) &= (b-a)\ln x+ \frac{b(b-1)-a(a-1)}{2x} - \frac{b^2(2b-3)-a^2(2a-3)}{12x^2} \\
&\qquad\qquad\qquad+ O(x^{-3})_{x \nearrow +\infty}\\
&= (b-a)\ln x + \frac{ (b-a)(b+a-1)}{2x} - \frac{(b-a)(2(b^2+ba+a^2)-3(b+a))}{12x^2}\\
&\qquad\qquad\qquad+ O(x^{-3})_{x \nearrow +\infty}. 
\end{split}\]
Consequently, we obtain
\begin{equation}\label{GammaBA-1}
\begin{split}
\frac{\Gamma(x+b)}{\Gamma(x+a)} &= x^{b-a} \exp \Big( \frac{ (b-a)(b+a-1)}{2x} - \frac{(b-a)(2(b^2+ba+a^2)-3(b+a))}{12x^2}+ O(x^{-3})_{x \nearrow +\infty} \Big).
\end{split}
\end{equation}
For any $\alpha$ and $ \beta$, we have
\[\begin{split}
\exp \Big( \frac{\alpha}x - \frac{\beta}{x^2}+ O(x^{-3})_{x \nearrow +\infty} \Big) 
&=  1+ \Big( \frac{\alpha}x-\frac{\beta}{x^2}\Big) + \frac12 \Big( \frac{\alpha}x-\frac{\beta}{x^2}\Big) ^2+ O(x^{-3})_{x \nearrow +\infty}\\
&= 1+ \frac{\alpha}x+ \frac{\alpha^2-2\beta}{2x^2}+ O(x^{-3})_{x \nearrow +\infty},
\end{split}\]
that implies
\begin{equation}\label{GammaBA-2}
\begin{split}
&\exp \Big( \frac{ (b-a)(b+a-1)}{2x} - \frac{(b-a)(2(b^2+ba+a^2)-3(b+a))}{12x^2}+ O(x^{-3})_{x \nearrow +\infty} \Big) \\
&= 1 +  \frac{ (b-a)(b+a-1)}{2x} + \frac{(b-a)(3(b-a)(b+a-1)^2-4(b^2+ba+a^2)+6(b+a))}{12 x^2}\\
&\qquad\qquad\qquad\qquad\qquad\qquad+ O(x^{-3})_{x \nearrow +\infty}. 
\end{split}
\end{equation}
Combining \eqref{GammaBA-1} and \eqref{GammaBA-2} gives us our desired formula \eqref{FractionGamma}.


\section{The exact value of $\overline\lambda(x)$ for any $x \in \R^n\setminus \{0\}$}
\label{apd-lambda}

We discuss the exact value of $\overline\lambda(x)$ at any $x \in \R^n\setminus \{0\}$ in this appendix. It follows from Lemma \ref{LeM=002} that
\[
\overline\lambda (x) \geq |x| \quad \text{for any }x \in \R^n \setminus \{0\}. 
\]
But the above inequality is enough to obtain the symmetry of $u$; see Lemma \ref{lem-MS-stop}. Consequently, we conclude the main result saying that $v$ must be constant. Remarkably, with help of the main result we are able to obtain the exact value of $\overline \lambda$. We shall prove the following.

\begin{proposition}
There holds
\[
\overline\lambda (x) = \sqrt{1+|x|^2} \quad \text{for any }x \in \R^n \setminus \{0\},
\]
where $\overline\lambda (x)$ is given by \eqref{eq-lambda}.
\end{proposition}
	
The proof goes as follows. By means of Theorem \ref{thm-Liouville} we know that the solution $v$ to \eqref{eq-MAIN} on $\S^n$ must be a non-negative constant, say $C$. Since $v$ is non-trivial, we also have $C > 0$. Thanks to \eqref{eq-u=v}, namely
\[
		u = \big( \frac 2{1+|x|^2} \big)^{\frac {n-2s}2} (v \circ \pi_N^{-1}),
\]
we conclude that
\[
u(x) = C \big( \frac 2{1+|x|^2} \big)^{\frac {n-2s}2} \quad\text{everywhere in } \R^n.
\]	
Next we assume by way of contradiction that 
\[
\overline \lambda (x) <\sqrt{1+|x|^2} \quad \text{for some } x \in \R^n.
\]
Then one can find sufficiently small $\ve >0$ in such a way that
\[
\big((\overline \lambda (x) + \ve)^2 - |y-x|^2\big) \big((\overline \lambda (x) + \ve) ^ 2 - 1 -|x|^2 \big)\geq 0
\]
for all $y$ satisfying $|y-x| \geq \overline \lambda (x) + \ve$. In fact, there holds
\[
\big(\lambda^2 - |y-x|^2\big) ( \lambda ^ 2 - 1 -|x|^2 )\geq 0
\]
for all $y$ satisfying $|y-x| \geq \lambda$ with all $\lambda \leq \overline \lambda (x) + \ve$. But this inequality is equivalent to
\[
\big(\frac { \lambda }{|y-x| } \big)^{n-2s} \Big( \frac 1{1+\Big|x+ \frac { \lambda^2 (y-x)}{|y-x|^2} \Big|^2 } \Big)^\frac{n-2s}2	 \leq \Big( \frac 1{1+ |y|^2 } \Big)^\frac{n-2s}2,
\]
which is nothing but
\[
u_{x,\lambda} (y) \leq u(y) \quad \text{for all } |y-x| \geq \lambda
\]
with $\lambda \leq \overline \lambda (x) + \ve$. But this contradicts the definition of $\overline\lambda (x) $ in \eqref{eq-lambda}. Thus, $\overline\lambda (x) = \sqrt{1+|x|^2}$ as claimed.

\end{document}